\PassOptionsToPackage{dvipsnames}{xcolor}
\documentclass[leqno]{siamart1116}

\usepackage{adjustbox}
\usepackage{algpseudocode}\usepackage[numbers]{natbib}
\usepackage{amssymb}
\usepackage{bm,bbm}
\usepackage{booktabs}
\usepackage{dsfont}
\usepackage{enumerate}
\usepackage{enumerate}
\usepackage{graphicx}
\usepackage{multirow}
\usepackage{xspace}
\usepackage{caption}
\usepackage[normalem]{ulem}
\usepackage{subfig}
\usepackage{gnuplot-lua-tikz}

\usepackage{mydef}

\let\cite=\citet

\begin{document}

\newcommand\footnotemarkfromtitle[1]{%
\renewcommand{\thefootnote}{\fnsymbol{footnote}}%
\renewcommand{\thefootnote}{\arabic{footnote}}}

\title{First-order greedy invariant-domain preserving approximation for hyperbolic problems: scalar conservation laws, and p-system}

\author{Jean-Luc Guermond\footnotemark[1] %
  \and Matthias Maier\footnotemark[1] %
  \and Bojan Popov\footnotemark[1]~\footnotemark[4] %
  \and Laura Saavedra\footnotemark[2] %
  \and Ignacio Tomas\footnotemark[3]}

\maketitle

\renewcommand{\thefootnote}{\fnsymbol{footnote}}
\footnotetext[1]{%
  Department of Mathematics, Texas A\&M University, 3368 TAMU, College
  Station, TX 77843, USA.}
\footnotetext[2]{%
  Departamento de Matem\'atica Aplicada a la Ingenier\'ia Aeroespacial,
  E.T.S.I. Aeron\'autica y del Espacio, Universidad Polit\'ecnica de Madrid,
  28040 Madrid, Spain}
\footnotetext[3]{%
  Department of Mathematics and Statistics, Texas Tech University, 2500
  Broadway Lubbock, TX 79409, USA.}
\footnotetext[4]{%
  Department of Mathematics and Informatics, University of Sofia, 5 James Boucher Blvd., 1164 Sofia, Bulgaria}

\renewcommand{\thefootnote}{\arabic{footnote}}

\begin{abstract}
  The paper focuses on first-order invariant-domain preserving
  approximations of hyperbolic systems. We propose a new way to estimate
  the artificial viscosity that has to be added to make explicit,
  conservative, consistent numerical methods invariant-domain preserving
  and entropy inequality compliant. Instead of computing an upper bound on
  the maximum wave speed in Riemann problems, we estimate a minimum wave
  speed in the said Riemann problems such that the approximation satisfies
  predefined invariant-domain properties and predefined entropy
  inequalities. This technique eliminates non-essential fast waves from the
  construction of the artificial viscosity, while preserving pre-assigned
  invariant-domain properties and entropy inequalities.
\end{abstract}

\begin{keywords}
 Conservation equations, hyperbolic systems, invariant domains, convex
  limiting, finite element method.
\end{keywords}

\begin{AMS}
  35L65, 65M60, 65M12, 65N30
\end{AMS}

\pagestyle{myheadings}
\thispagestyle{plain}
\markboth{J.-L.\,Guermond, M.\,Maier, B.\,Popov, L.\,Saavedra, I.\,Tomas}%
  {Greedy invariant domain preserving approximation for hyperbolic systems}

%%%%%%%%%%%%%%%%%%%%%%%%%%%%%%%%%%%%%%%%%%%%%%%%%%%%%%%%%%%%%%%%%%%%%%%%%%%%%%%%
%%%%%%%%%%%%%%%%%%%%%%%%%%%%%%%%%%%%%%%%%%%%%%%%%%%%%%%%%%%%%%%%%%%%%%%%%%%%%%%%
  %%%%%%%%%%%%%%%%%%%%%%%%%%%%%%%%%%%%%%%%%%%%%%%%%%%%%%%%%%%%%%%%%%%%%%%%%%%%%%%%
  
\section{Introduction}
\label{Sec:introduction}

Let us consider the hyperbolic system of conservation equations
$\partial_t\bu+\DIV\polf(\bu)=\bzero$, where $\bu$ denotes a conserved
state taking values in $\Real^m$ and $\polf(\bu)$ an associated flux
taking values in $\Real^{m\CROSS d}$, where $d$ is the space
dimension. Most explicit approximation methods for solving this type
of system are based on some notion of numerical flux and involve some
numerical dissipation. For instance, all the first-order methods based on Lax's
seminal paper~\citep[p.~163]{Lax_1954} involve numerical fluxes
between pairs of degrees of freedom, say $i,j$, that take the
following form
$\frac12 (\polf(\bsfU_i)+\polf(\bsfU_j))\bn_{ij} + \alpha_{ij}
(\bsfU_i-\bsfU_j)$, where $\bn_{ij}$ is some unit vector associated
with the space discretization at hand and $\alpha_{ij}$ is an upper
bound on the maximum wave speed in the Riemann problem using the flux
$\polf(\bsfU)\bn_{ij}$ and the states $\bsfU_i$ and $\bsfU_j$ as left
and right Riemann data. Denoting by
$\lambda_{\max}(\bn_{ij},\bsfU_i,\bsfU_j)$ the maximum wave speed in
the Riemann problem in question, it is now well established that
choosing $\alpha_{ij}$ such that $\alpha_{ij} \ge
\lambda_{\max}(\bn_{ij},\bsfU_i,\bsfU_j)$ guarantees that some invariant-domain
property can be extracted from the scheme; see \eg
\cite{Harten_Lax_VanLeer_1983}, \cite[p.~375]{Tadmor_1984},
\cite[\S5]{Perthame_Shu_1996}. Using $\lambda_{\max}(\bn_{ij},\bsfU_i,\bsfU_j)$
to construct invariant-domain preserving schemes dates back to the origins of
computational fluid dynamics; we refer the reader for instance
to \citep[p.~163]{Lax_1954}. Recalling that the
flux $\alpha_{ij}(\bsfU_i-\bsfU_j)$ is associated with numerical dissipation
and therefore induces a loss of accuracy, a natural question to ask is whether
it is possible to estimate a greedy value for $\alpha_{ij}$ in the open interval
$(0, \lambda_{\max}(\bn_{ij},\bsfU_i,\bsfU_j))$ guaranteeing that the
scheme satisfies the desired invariant-domain properties and relevant
entropy inequalities. It is the purpose of the present paper to
give a positive answer to this question. The paper is the result of a research
project that was initiated at the 9th International Conference on Numerical
Methods for Multi-Material Fluid Flow, held in Trento, Italy, 9-13, September
2019. Some of the questions posed above and some answers thereto were outlined
in \citep{OSTIslides}.

To convince the reader that the program described above is feasible,
let us consider the compressible Euler equations equipped with a
$\gamma$-law, and let us consider the Riemann problem with the flux
$\polf(\bu)\SCAL \bn$ and some left and right data
$\bu_L, \bu_R$. Then denoting by $\lambda_1^-, \lambda_1^+$ the
two wave speeds enclosing the $1$-wave, $\lambda_2$ the speed of the
$2$-wave (\ie the contact discontinuity), and
$\lambda_3^-, \lambda_3^+$ the two wave speeds enclosing the $3$-wave,
we have
$\lambda_1^-\le \lambda_1^+\le \lambda_2\le \lambda_3^-\le
\lambda_3^+$, and the maximum wave speed in the Riemann problem is
$\lambda_{\max}(\bn,\bu_L,\bu_R):=\max(|\lambda_1^-|,|\lambda_3^+|)>|\lambda_2|$.
If the Riemann data yields a solution that consists of just one contact
discontinuity, one can establish that the amount of viscosity that is
sufficient to satisfy all the invariant domain properties (in addition
to local entropy inequalities) is just $|\lambda_2|$ (because the
velocity and the pressure are constant in this case). Hence setting
the graph viscosity wave speed $\alpha$ to be larger than or equal to
$\lambda_{\max}(\bn,\bu_L,\bu_R)$ is needlessly over-diffusive
since taking $\alpha=|\lambda_2|$ is sufficient in this case. Invoking
the continuous dependence of the Riemann solution with respect to the
data, one then realizes that a similar conclusion holds if the Riemann
data is a small perturbation of a data set producing a contact
discontinuity only. The situation described above is well illustrated
by the multi-material Euler equations in Lagrangian coordinates. In
this case the interface between two materials is a contact
discontinuity that should keep its integrity over time. Let
$v:=\bu\SCAL \bn$ denote the component of the material velocity normal
to the interface. The maximum wave speed in the Riemann problem using
the two states on either sides of the interface gives
$\lambda_{\max}(\bn,\bu_L,\bu_R)=\max(|\lambda_1^--v|,|\lambda_3^+-v|)$
in the Lagrangian reference frame, whereas the wave speed of the
$2$-wave is $\lambda_2-v =0$. In this case the amount of viscosity
that is sufficient to satisfy all the invariant domain properties is
$\alpha=\lambda_2-v =0$. Hence, if one instead uses
$\alpha=\lambda_{\max}:=\max(|\lambda_1^--v|,|\lambda_3^+-v|)$ (as
suggested \eg in \cite{Guermond_Popov_Saavedra_Yong_2017} and most of
the literature on the topic) one needlessly diffuses the contact
discontinuity. The purpose of the present paper is to clarify the
issues described above and derive a variation of the method presented in
\citep{Guermond_Popov_SINUM_2016,Guermond_Popov_Saavedra_Yong_2017}
that is invariant-domain preserving, satisfies discrete entropy
inequalities, and minimizes the amount of artificial viscosity used.

The first-order method presented in the paper can be made
high-order and still be invariant-domain preserving by using one of many
techniques developed to this effect and available in the abundant
literature dedicated to the topic. This can be done by adapting the flux
transport corrected methodology from \cite[\S{II}]{Zalesak_1979}. For
instance, one can use methods inspired from \cite{TurekKuzmin2002,
KuzminLoehnerTurek2012} when the functionals to limit are affine. When
these functionals are nonlinear, one can use methods from
\cite{Zhang_Shu_2010, Zhang_Shu_2011} (for discontinuous finite elements)
or from \citep{Guermond_Nazarov_Popov_Tomas_SISC_2019,
Guermond_Popov_Tomas_CMAME_2019} (for continuous finite elements). A
complete list of all the excellent methods capable of achieving this goal
cannot be cited here.

The paper is organized as follows. We formulate the problem and recall
important concepts that are used in the paper in
\S\ref{Sec:Formulation_of_the_problem}. We introduce  in
\S\ref{Sec:General_strategy} the concept of
greedy viscosity for any hyperbolic system.
The key results of this section are the
definitions \eqref{def_of_greedy_viscosity} and \eqref{def_dij_greedy}
and Theorem~\ref{Thm:UL_is_invariant_greedy}. The concept of greedy
viscosity is then illustrated for scalar conservation equations in
\S\ref{Sec:Scalar_equations}. The main result summarizing the content
of this section are the definitions \eqref{def_lambda_12},
\eqref{Lambda_Scalar_Entrop_Ineq} and
Theorem~\ref{Thm:UL_is_invariant_modified}. The concept is further
illustrated for the $p$-system in \S\ref{Sec:p_system}.
The ideas
introduced in the paper are numerically illustrated in
\S\ref{Sec:Numerical_illustrations} for scalar conservation equations
and for the $p$-system. Some of these tests are meant to illustrate
that estimating a greedy wave speed in order to preserve the
invariant-domain is not sufficient to converge to an entropy
solution. Ensuring that entropy inequalities are satisfied is
essential for this matter. We also show that using just one entropy is
not sufficient for scalar conservation equations with a
non-convex flux. Due to lack of space, the concept of greedy viscosity
for systems like the compressible Euler equations equipped with a tabulated
equation of state will be illustrated in a forthcoming second part of this
work. A short outline of the performance of the method is given in the
conclusions section, see \S\ref{Sec:conclusions}.

%%%%%%%%%%%%%%%%%%%%%%%%%%%%%%%%%%%%%%%%%%%%%%%%%%%%%%%%%%%%%%%%%%%%%%%%%%%%%%%%
%%%%%%%%%%%%%%%%%%%%%%%%%%%%%%%%%%%%%%%%%%%%%%%%%%%%%%%%%%%%%%%%%%%%%%%%%%%%%%%%
%%%%%%%%%%%%%%%%%%%%%%%%%%%%%%%%%%%%%%%%%%%%%%%%%%%%%%%%%%%%%%%%%%%%%%%%%%%%%%%%

\section{Formulation of the problem}
\label{Sec:Formulation_of_the_problem}
In this section we formulate the question that is addressed in the paper and
put it in context.

%%%%%%%%%%%%%%%%%%%%%%%%%%%%%%%%%%%%%%%%%%%%%%%%%%%%%%%%%%%%%%%%%%%%%%%%%%%%%%%%
\subsection{The hyperbolic system}
\label{Sec:hyperbolic_system}
Our objective is to develop elementary and robust numerical tools to
approximate hyperbolic systems in conservation form:
\begin{equation}
  \label{def:hyperbolic_system}
  \begin{cases} \partial_t \bu + \DIV \polf(\bu)=\bzero,
    \quad \mbox{for}\, (\bx,t)\in \Dom\CROSS\Real_+,\\
    \bu(\bx,0) = \bu_0(\bx), \quad \mbox{for}\, \bx\in \Real^d.
  \end{cases}
\end{equation}
Here $d$ is the space dimension, $\Dom$ is a compact, connected, polygonal
subset of $\Real^d$. To avoid difficulties related to boundary conditions,
we either solve the Cauchy problem or assume that the boundary conditions
are periodic. The dependent variable (or state variable) $\bu$ takes values
in $\Real^m$. The function $\polf:\calA\to (\Real^m)^d$ is called flux.
The domain of $\polf$, \ie $\calA\subset\Real^{m}$, is called
\emph{admissible set}. The state variable $\bu$ is viewed as a column
vector $\bu=(u_1,\ldots,u_m)\tr$. The flux is a $m\CROSS d$ matrix with
entries $\polf_{ik}(\bu(\bx))$, $i\in\intset{1}{m}$, $k\in\intset{1}{d}$
and $\DIV\polf(\bu(\bx))$ is a column vector with entries
$(\DIV\polf(\bu))_i= \sum_{k\in\intset{1}{d}}\partial_{x_k}
\polf_{ik}(\bu(\bx))$. For any $\bn=(n_1\ldots,n_d)\tr\in \Real^d$, we
denote $\polf(\bu)\bn$ the column vector with entries
$\sum_{l\in\intset{1}{d}} \polf_{il}(\bu) n_l$, where $i\in\intset{1}{m}$.

We assume in the entire paper that the admissible set $\calA
\subset\Real^{m}$ is constructed such that for every pair of states
$(\bu_L,\bu_R)\in \calA\CROSS \calA$ and every unit vector $\bn$ in
$\Real^d$, the following one-dimensional Riemann problem
\begin{equation}
  \label{def:Riemann_problem}
  \partial_t \bw + \partial_x (\polf(\bw)\bn)=0,
  \quad (x,t)\in \Real\CROSS\Real_+,\qquad
  \bw(x,0) =
  \begin{cases}
    \bu_L, & \text{if $x<0$} \\
    \bu_R, & \text{if $x>0$},
  \end{cases}
\end{equation}
has a unique solution satisfying adequate entropy inequalities. We
assume that this solution is self-similar with self-similarity
parameter $\xi \eqq \frac{x}{t}$, and we set
\begin{equation}
  \bv(\bn,\bu_L,\bu_R,\frac{x}{t})\eqq \bw(x,t);
\end{equation}
see for instance \cite{Lax_1957_II,Toro_2009}. Using Lax's notation, we
denote by $\lambda_1^-\le \lambda_1^+$ the two wave speeds enclosing the
$1$-wave (\ie the leftmost wave) and $\lambda_m^-\le \lambda_m^+$ the two
wave speeds enclosing the $m$-wave (\ie the rightmost wave). The key result
that we are going to use in the paper is that $\bv(\bn,\bu_L,\bu_R, \xi) =
\bu_L$ if $\xi\le\lambda_1^-$ and $\bv(\bn,\bu_L,\bu_R, \xi) = \bu_R$ if
$\xi\ge \lambda_m^+$. We define a \emph{left wave speed}
$\lambda_{L}(\bn,\bu_L,\bu_R)\eqq \lambda_1^-$ and a \emph{right wave speed}
$\lambda_{R}(\bn,\bu_L,\bu_R)\eqq\lambda_m^+$. We also define
\emph{the maximum wave speed} of the Riemann problem to be
\begin{align}
  \label{eq:lambda_max}
  \lambda_{\max}(\bn,\bu_L,\bu_R) \eqq
  \max(|\lambda_{L}(\bn,\bu_L,\bu_R)|,|\lambda_{R}(\bn,\bu_L,\bu_R)|).
\end{align}
We will replace the notation $\lambda_{\max}(\bn,\bu_L,\bu_R)$ by
$\lambda_{\max}$ when the context is unambiguous. For further reference,
for every $t>0$ we define
\begin{equation}
  \label{Riemann_average}
  \overline\bv(t,\bn,\bu_L,\bu_R) :=\int_{-\frac12}^{\frac12}
  \bv(\bn,\bu_L,\bu_R,\tfrac{x}{t}) \diff x.
\end{equation}
Notice that if $t\lambda_{\max}(\bn,\bu_L,\bu_R) \le \frac12$, then
$\overline\bv(t,\bn,\bu_L,\bu_R)$ is the average of the entropy solution of
the Riemann problem \eqref{def:Riemann_problem} over the Riemann fan. This
property is illustrated in Figure~\ref{Fig:riemann_problem}.

\begin{figure}[t]
  \begin{center}
    \adjustbox{scale=0.35,trim =0 20 0 0}{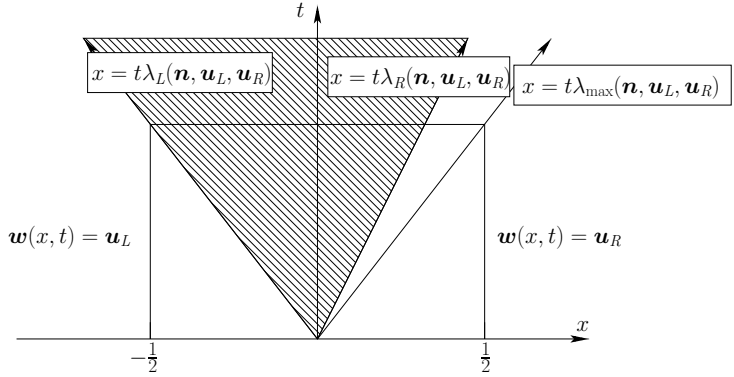}
  \end{center}
  \caption{Riemann problem and Riemann fan.} \label{Fig:riemann_problem}
\end{figure}

\begin{definition}[Invariant domain]
  \label{Def:invariant_set}
  We say that $\calB\subset \calA\subset \Real^m$ is invariant domain
  for \eqref{def:hyperbolic_system} if the following holds true:
  \textup{(i)} $\calB$ is convex; \textup{(ii)} for any pair
  $(\bu_L,\bu_R)\in \calB\CROSS \calB$, any unit vector
  $\bn\in \Real^d$, and all
  $t\in\big(0,\frac{1}{2\lambda_{\max}(\bn,\bu_L,\bu_R)}\big)$, we
  have $\overline\bv(t,\bn,\bu_L,\bu_R)\in \calB$, where
  $\overline\bv$ is given by \eqref{Riemann_average}.
\end{definition}

\begin{lemma}[Invariance of the auxiliary states]\label{Lem:InvarBar}
  Let $\calB\subset \calA$ be any invariant domain
  for~\eqref{def:hyperbolic_system}. Let $(\eta,\bq)$ be an entropy
  pair for~\eqref{def:hyperbolic_system}. Let $\lambda>0$, let $\bn\in\Real^d$ be a
  unit vector. For all $\bu_L$, $\bu_R$ in $\calA$, consider the following
  auxiliary state:
  \begin{align}
    \label{def_generic_barstates}
    \overline{\bu}_{LR}(\lambda) :=
    \frac{1}{2}(\bu_L + \bu_R)
    -\frac{1}{2\lambda}\left(\polf(\bu_R) - \polf(\bu_L)\right)\bn.
  \end{align}
  Assume that $\bu_L,\bu_R\in \calB$, and
  $\lambda\ge \lambda_{\max}(\bn,\bu_L,\bu_R)$. Then
  \begin{align}
    &\overline{\bu}_{LR}(\lambda) =
    \overline\bv\left(\tfrac{1}{2\lambda},\bn,\bu_L,\bu_R\right),
    \label{eq1:Lem:InvarBar}
    \\
    &\overline{\bu}_{LR} (\lambda)\in \calB, \label{eq2:Lem:InvarBar}
    \\
    &\eta(\overline\bu _{LR} (\lambda)) \le
    \tfrac{1}{2}(\eta(\bu_L)+\eta(\bu_R)) -
    \tfrac{1}{2\lambda}\left(\bq(\bu_R) - \bq(\bu_L)\right)\SCAL\bn
    \label{eq3:Lem:InvarBar}
  \end{align}
\end{lemma}
\begin{proof}
  See \eg Lemma~2.1 and Lemma~2.2 in \citep{Guermond_Popov_SINUM_2016}.
\end{proof}

%%%%%%%%%%%%%%%%%%%%%%%%%%%%%%%%%%%%%%%%%%%%%%%%%%%%%%%%%%%%%%%%%%%%%%%%%%%%%%%%
\subsection{Agnostic space approximation}
Without going into details, we now assume that we have at hand a fully
discrete scheme where time is approximated by using the forward Euler time
stepping and space is approximated by using some ``centered'' approximation
of \eqref{def:hyperbolic_system}, \ie without any artificial viscosity to
stabilize the approximation. We denote by $t^n$ the current time, $n\in
\polN$, and we denote by $\dt$ the current time step size; that is
$t^{n+1}:=t^n+\dt$ (we should write $\dt^n$ as the time step may vary at each time step,
but we omit the super-index $^n$ to simplify the notation). Let us assume that the current approximation is a
collection of states $\{\bsfU_i\upn\}_{i\in\calV}$, where the index set
$\calV$ is used to enumerate all the degrees of freedom of the
approximation. We assume that the ``centered'' update is given by
$\bsfU_i\upGnp$ with
\begin{equation}
  \label{Galerkin_approx}
  \frac{m_i}{\dt}(\bsfU_i\upGnp-\bsfU_i^n) +\!\!\sum_{j\in \calI(i)}
  \polf(\bsfU_j^n) \bc_{ij} = \bzero.
\end{equation}
The quantity $m_i$ is called lumped mass and we assume that $m_i>0$ for all
$i\in\calV$. The index set $\calI(i)$ is called local stencil. This set
collects only the degrees of freedom in $\calV$ that interact with
$i$.
We set $\calI(i)^*\eqq\calI(i){\setminus}\{i\}$. The
vector $\bc_{ij}\in \Real^d$ encodes the space discretization. We view
$\frac{1}{m_i}\sum_{j\in \calI(i)}\polf(\bsfU_j^n) \bc_{ij}$ as a Galerkin
(or centered or inviscid) approximation of $\DIV\polf(\bu)$ at time $t^n$
at some grid point (or cell) $i\in\calV$. The superscript $^\textup{G}$ is
meant to remind us that \eqref{Galerkin_approx} is a Galerkin (or inviscid
or centered) approximation of \eqref{def:hyperbolic_system}. That is, we
assume that the consistency error in space in \eqref{Galerkin_approx}
scales optimally with respect to the mesh size for the considered
approximation setting. We keep the discussion at this abstract level for
the sake of generality; see Remark~\ref{rem:literature}. The only
requirement that we make on the coefficients $\bc_{ij}$ is that the method
is conservative; that is to say, we assume that
\begin{align}\label{cijprop}
  \bc_{ij} = - \bc_{ji}\quad \text{and} \quad \sum_{j\in \calI(i)} \bc_{ij}
  = \bzero.
\end{align}
An immediate consequence of this assumption is that the total mass is
conserved: $\sum_{i\in\calV} m_i \bsfU_i\upGnp= \sum_{i\in\calV} m_i
\bsfU_i\upn$.

Of course, \eqref{Galerkin_approx} is in general not appropriate if the
solution to~\eqref{def:hyperbolic_system} is not smooth. To recover some
sort of stability (the exact notion of stability we have in mind is defined
in Theorem~\ref{Thm:UL_is_invariant}) we modify the scheme by adding a
graph viscosity based on the stencil $\calI(i)$; that is, we compute the
stabilized update $\bsfU_i^{n+1}$ by setting:
\begin{equation}
  \label{low_order_scheme}
  \frac{m_i}{\dt}(\bsfU_i^{n+1}-\bsfU_i^n) +\!\!\sum_{j\in \calI(i)}
  \polf(\bsfU_j^n) \bc_{ij} - \!\!\sum_{j\in \calI(i)^*}
  d_{ij}\upn (\bsfU^n_j - \bsfU^n_i)  = \bzero.
\end{equation}
Here $d_{ij}\upn$ is the yet to be defined graph viscosity. We assume that
\begin{align}\label{dijprop}
  d_{ij}\upn = d_{ji}\upn> 0, \quad \text{if} \quad i\ne j.
\end{align}
The symmetry implies that the method remains conservative. The question addressed in
the paper is the following: how large has $d_{ij}\upn$ to be chosen so
that \eqref{low_order_scheme} preserves invariant domains and satisfies
entropy inequalities (for some finite collection of entropies)?

\begin{remark}[Literature]
  \label{rem:literature} The algorithm \eqref{low_order_scheme} is a
  generalization of \citep[p.~163]{Lax_1954}; see also
  \cite{Harten_Lax_VanLeer_1983}, \cite[p.~375]{Tadmor_1984},
  \cite[\S5]{Perthame_Shu_1996} and the literature cited in these references.
  The reader is referred to \citep{Guermond_Popov_SINUM_2016},
  \citep{Guermond_Nazarov_Popov_Tomas_SISC_2019} for realizations of the above
  algorithm with continuous finite elements. Realizations of the scheme
  with finite volumes and discontinuous elements
  are described in \citep{Guermond_Popov_Tomas_CMAME_2019} and implemented in
  \cite{TomasMaierShad2023, TomasKronMaier2024}.
\end{remark}

%%%%%%%%%%%%%%%%%%%%%%%%%%%%%%%%%%%%%%%%%%%%%%%%%%%%%%%%%%%%%%%%%%%%%%%%%%%%%%%%
\subsection{The  auxiliary bar states}
\label{Sec:aux_bar_states}
We now recall the main stability result established
in~\citep{Guermond_Popov_SINUM_2016}. The proof of this result is the
source of inspiration for the rest of the paper. For all $i\in\calV$ and
all $j\in\calI(i)$ we introduce the unit vector $\bn_{ij}:=
\bc_{ij}/\|\bc_{ij}\|_{\ell^2}$. Given two states $\bsfU_i^n$ and
$\bsfU_j^n$ in $\calA$, we recall that
$\lambda_{\max}(\bn_{ij},\bsfU_i^n,\bsfU_j^n)$ is the maximum wave speed in
the Riemann problem  defined in~\S\ref{Sec:hyperbolic_system} with left
state $\bsfU_i^n$, right state $\bsfU_j^n$, and unit vector $\bn_{ij}$. The
guaranteed maximum speed (GMS) graph viscosity $d_{ij}\upGMSn$ is defined
in \citep{Guermond_Popov_SINUM_2016} as follows:
\begin{equation}
  \label{Def_of_dij}
  d_{ij}\upGMSn := \max\big(\lambda_{\max}(\bn_{ij},\bsfU_i^n,\bsfU_j^n)
  \|\bc_{ij}\|_{\ell^2},
  \lambda_{\max}(\bn_{ji},\bsfU_j^n,\bsfU_i^n) \|\bc_{ji}\|_{\ell^2}\big).
\end{equation}

\begin{theorem}[Local invariance]
  \label{Thm:UL_is_invariant}
  Let $\calB\subset \calA$ be any invariant domain for
  \eqref{def:hyperbolic_system}. Let $(\eta,\bq)$ be any entropy pair
  for~\eqref{def:hyperbolic_system}. Let $n\ge 0$ and $i\in\calV$. Let
  $d_{ij}^n$ be any graph viscosity such that $d_{ij}^n\ge d_{ij}\upGMSn$
  and $d_{ij}^n>0$. Assume that $0<\dt\le m_i/\sum_{j\in\calI(i)^*}2d_{ij}^n$.
  Assume that $\{\bsfU_j^n\}_{j\in \calI(i)}\subset \calB$. Then the update
  $\{\bsfU_i\upnp\}_{i\in\calV}$ given by \eqref{low_order_scheme}
  satisfies the following properties:
  \begin{align}
    \label{eq1:Thm:UL_is_invariant}
    &\bsfU_i\upnp \in \calB,
    \\
    \label{eq2:Thm:UL_is_invariant}
    &\frac{m_i}{\dt} (\eta(\bsfU_i\upnp) - \eta(\bsfU_i^n))
    + \sum_{j\in\calI(i)} \bc_{ij}\SCAL \bq(\bsfU_j^n) -\!\!\!
    \sum_{j\in \calI(i)^*} \!\!\! d_{ij}^n
    (\eta(\bsfU_{j}^{n})  - \eta(\bsfU_{j}^{n})) \le 0.
  \end{align}
\end{theorem}

\begin{proof}
  We refer to Theorem~4.1 and Theorem~4.5 in~\citep{Guermond_Popov_SINUM_2016}
  for detailed proofs. But since these proofs contain ideas that are going
  to be used latter in the paper, we now reproduce the key arguments. Using
  the conservation property \eqref{cijprop}, \ie $\sum_{j\in \calI(i)}
  \bc_{ij} = \bzero$, we rewrite~\eqref{low_order_scheme} as follows:
  \begin{equation*}
    \frac{m_i}{\dt}(\bsfU_i\upnp-\bsfU_i^n) +\!\!\sum_{j\in
    \calI(i)^*} \Big(2d_{ij}\upn \bsfU^n_i
    + ( \polf(\bsfU_j^n) - \polf(\bsfU_i^n) ) \bc_{ij}
    - d_{ij}\upn (\bsfU^n_j + \bsfU^n_i)\Big)  = \bzero.
  \end{equation*}
  Then, recalling that $d_{ij}\upn>0$ by assumption, we introduce the
  auxiliary states
  \begin{align}
    \label{def_barstates}
    \overline{\bsfU}_{ij}\upn :=
    \frac{1}{2}(\bsfU^{n}_i + \bsfU^{n}_j)
    -(\polf(\bsfU_j^n) - \polf(\bsfU_i^n))\bn_{ij}
    \frac{\|\bc_{ij}\|_{\ell^2}}{2 d_{ij}\upn}.
  \end{align}
  This allows us to rewrite~\eqref{low_order_scheme} as follows:
  \begin{align}
    \label{def_dij_scheme_convex}
    \bsfU_i\upnp = \bigg(1 -\!\! \sum_{j \in \calI(i)^*}
    \frac{2 \dt d_{ij}\upn}{m_i} \bigg)
    \bsfU_{i}^{n} + \!\! \sum_{j \in \calI(i)^*} \frac{2 \dt
    d_{ij}\upn}{m_i} \overline{\bsfU}_{ij}\upn.
  \end{align}
  Since we assumed that $1-2\frac{\dt}{m_i}\sum_{j\in\calI(i)^*}d_{ij}^n>
  0$, the right-hand side in the above identity is a convex combination of
  the states $\{\overline{\bsfU}_{ij}\upn\}_{j\in \calI(i)}$ with the
  convention $\overline{\bsfU}_{ii}\upn := \bsfU^{n}_i$. Setting
  $\lambda_{ij}:= \frac{d_{ij}^n}{\|\bc_{ij}\|_{\ell^2}}$ and recalling
  definition~\eqref{def_generic_barstates}, we observe that
  $\overline{\bsfU}_{ij}\upn = \overline{\bu}_{ij}(\lambda_{ij})$ (here, with
  slight abuse of notation, we write $\overline{\bu}_{ij}(\lambda)$ instead
  of $\overline{\bu}_{LR}(\lambda)$). Then the assumption $d_{ij}^n\ge
  d_{ij}\upGMSn$ implies that $\lambda_{ij} \ge
  \lambda_{\max}(\bn_{ij},\bsfU_i^n,\bsfU_j^n)$, and the rest of the proof
  readily follows by invoking Lemma~\ref{Lem:InvarBar} (in particular
  $\overline{\bsfU}_{ij}\upn = \overline{\bu}_{ij}(\lambda_{ij}) =
  \overline\bv(\frac{1}{2\lambda_{ij}},\bn_{ij},\bsfU_i^n,\bsfU_j^n)$).
\end{proof}

\begin{remark}[$\lambda_\epsilon$ and $\lambda_{\max}^\calV$]
  \label{Rem:sonic_point} \label{Rem:lambda_epsilon}
  The expression~\eqref{def_barstates} (and thereby the identity
  \eqref{def_dij_scheme_convex} as well) is ill-defined if
  $\lambda_{\max}(\bn_{ij},\bsfU_i^n, \bsfU_j^n)=0$, (recall that
  Lemma~\ref{Lem:InvarBar} requires that one should take
  $\lambda\ge \lambda_{\max}(\bn_{ij},\bsfU_i^n, \bsfU_j^n)$). To
  avoid the division by zero issue, we introduce a small number
  $\epsilon\in (0,1)$ and we define
  \begin{subequations}
    \begin{align}
      &\lambda_{\max}^\calV \eqq \max_{i\in\calV,
      j\in\calI(i)}\lambda_{\max}(\bn_{ij},\bsfU_i^n,\bsfU_j^n),\qquad
      \lambda_\epsilon \eqq \epsilon \lambda_{\max}^\calV,
      \label{def_lambda_0}
      \\
      &\lambda_{ij}^{\sharp}\eqq\max(\lambda_\epsilon,\lambda_{\max}(\bn_{ij},
      \bsfU_i^n,\bsfU_j^n)).
      \label{def_lambda_sharp}
    \end{align}
  \end{subequations}
  Henceforth we assume that $\lambda_{\max}^\calV>0$, which implies
  $\lambda_\epsilon>0$. Otherwise the wave speed is zero everywhere, the
  solution is constant in time, and there is nothing to update. We are now
  going to consider the auxiliary states $\overline\bu_{ij}(\lambda)$ and
  \eqref{def_barstates} for $\lambda\in
  [\lambda_\epsilon,\lambda_{ij}^{\sharp}]$.
\end{remark}

\begin{remark}[Key observation]
  \label{Rem:key_observations}
  The statements \eqref{eq1:Thm:UL_is_invariant} and
  \eqref{eq2:Thm:UL_is_invariant} in Theorem~\ref{Thm:UL_is_invariant}
  are consequences
  of~\eqref{eq2:Lem:InvarBar}-\eqref{eq3:Lem:InvarBar} in
  Lemma~\ref{Lem:InvarBar}. And the assertions
  \eqref{eq2:Lem:InvarBar}-\eqref{eq3:Lem:InvarBar} hold true because
  $\lambda\ge \lambda_{\max}(\bn,\bu_L,\bu_R)$ implies the
  identity~\eqref{eq1:Lem:InvarBar}, \ie
  $\overline{\bu}_{LR} =
  \overline\bv(\frac{1}{2\lambda},\bn,\bu_L,\bu_R)$. We note,
  though, that $\lambda\ge \lambda_{\max}(\bn,\bu_L,\bu_R)$ (and
  thus identity~\eqref{eq1:Lem:InvarBar}) is just a \emph{sufficient}
  condition for \eqref{eq2:Lem:InvarBar}-\eqref{eq3:Lem:InvarBar} to
  hold true. The remainder of the paper is dedicated to estimating a
  \emph{greedy} wave speed
  $\lambda_{LR}^{\textup{grdy}}\,\in\,[\lambda_\epsilon,
  \lambda_{LR}^{\sharp}]$ (depending on $\bn$, $\bu_L$ and
  $\bu_R$) that is as small as possible so that
  \eqref{eq2:Lem:InvarBar}-\eqref{eq3:Lem:InvarBar} still holds,
  although \eqref{eq1:Lem:InvarBar} may no longer hold. For this wave
  speed $\lambda_{ij}^{\textup{grdy}}$ all the assertions in
  Theorem~\ref{Thm:UL_is_invariant} still hold true after redefining
  the viscosity
  $d_{ij}^n:=
  \max\big(\lambda_{ij}^{\textup{grdy}}\|\bc_{ij}\|_{\ell^2},
  \lambda_{ji}^{\textup{grdy}}\|\bc_{ji}\|_{\ell^2}\big).$
  This minimization program is reasonable since in
  the worst case scenario setting $\lambda = \lambda_{LR}^\sharp\ge
  \lambda_{\max}(\bn,\bu_L,\bu_R)$ is always admissible, \ie the
  minimizing set for $\lambda$ is not empty.
\end{remark}

\begin{remark}[Literature] The importance of the auxiliary states
  $\overline{\bu}_{LR}(\lambda)$, which are the backbone of Lax's scheme,
  has been recognized in \cite[Eq.~(2.6)]{Nessyahu_Tadmor_1990}. That these
  states are averages of Riemann solutions provided $\lambda$ is larger than
  $\lambda_{\max}$ is well documented in
  \cite[\S3.A]{Harten_Lax_VanLeer_1983} (see also the reference to a
  private communication with Harten at p.~375, line 12 in
  \cite{Tadmor_1984}). A variant of Lemma~\ref{Lem:InvarBar} is invoked to
  prove Theorem~3.1 in \citep{Harten_Lax_VanLeer_1983}. This theorem is a
  somewhat simplified version of Theorem~\ref{Thm:UL_is_invariant}.
 \end{remark}

%%%%%%%%%%%%%%%%%%%%%%%%%%%%%%%%%%%%%%%%%%%%%%%%%%%%%%%%%%%%%%%%%%%%%%%%%%%%%%%%
%%%%%%%%%%%%%%%%%%%%%%%%%%%%%%%%%%%%%%%%%%%%%%%%%%%%%%%%%%%%%%%%%%%%%%%%%%%%%%%%
%%%%%%%%%%%%%%%%%%%%%%%%%%%%%%%%%%%%%%%%%%%%%%%%%%%%%%%%%%%%%%%%%%%%%%%%%%%%%%%%

\section{Greedy wave speed and greedy viscosity}
\label{Sec:General_strategy}
The key idea of the paper is introduced in this section. Let $\calB$ be a
convex invariant domain for \eqref{def:hyperbolic_system}. In this entire
section $\bn$ is a unit vector and $\bu_L,\bu_R$ are two states in
$\calB$. The important results of this section are the definitions
\eqref{def_of_greedy_viscosity}-\eqref{def_dij_greedy} and
Theorem~\ref{Thm:UL_is_invariant_greedy}. Owing to
Lemma~\ref{Lem:InvarBar}, we know that the invariant-domain property
\eqref{eq2:Lem:InvarBar} and the entropy
inequality~\eqref{eq3:Lem:InvarBar} hold for $\overline{\bu}_{LR}(\lambda)$
if $\lambda\ge \lambda_{\max}(\bn,\bu_L,\bu_R)$. Our objective in this
paper is to find a value of $\lambda$ as small as possible in the interval
$[\lambda_\epsilon, \lambda_{LR}^{\sharp}]$ so that
\eqref{eq2:Lem:InvarBar} and \eqref{eq3:Lem:InvarBar} still hold (we no
longer require that \eqref{eq1:Lem:InvarBar} be true). The actual
estimation of this greedy wave speed in done
\S\ref{Sec:def_greedy_wave_speed}.

%%%%%%%%%%%%%%%%%%%%%%%%%%%%%%%%%%%%%%%%%%%%%%%%%%%%%%%%%%%%%%%%%%%%%%%%%%%%%%%%
\subsection{Invariant domain and entropy: structural assumptions}
\label{Sec:Structural_assumptions}

As the notion of an invariant domain of the PDE
system~\eqref{def:hyperbolic_system} is too general, we list in this
section the properties that we want to preserve. We use the concept of
quasiconcavity for his purpose. (The reader who is not familiar with this notion
can replace the word \emph{quasiconcavity} by \emph{concavity} without losing
the essence of what is said.)
\begin{definition}[Quasiconcavity]\label{def:quasiconv} Given a convex set
  $\calC \subset \Real^m$, we say that a function $\Psi:\calC \to \Real$ is
  quasiconcave if the set $L_\chi(\Psi) := \{\bu\in \calC \st \Psi(\bu)
  \ge \chi \}$ is convex for every $\chi\in \Real$. The sets
  $\{L_\chi(\Psi)\}_{\chi\in \Real}$ are called \emph{upper level sets} or
  \emph{upper contour sets}.
\end{definition}

We now list the properties we are interested in and that we want to
preserve. Let $L\in\polN{\setminus}\{0\}$ and let us set $\calL\eqq
\intset{0}{L}$, $\calL^*\eqq \intset{1}{L}$. We assume that there
exists a collection of $L+1$ subsets $\{\calB_l\}_{l\in\calL}$ in
$\Real^m$, and a collection of $L$ continuous quasiconcave functionals
$\{\Psi_l:\calB_{l-1} \to \Real\}_{l\in\calL^*}$ so that the following
properties hold true:%
\begin{subequations}
  \label{Ass_Psi}%
  \begin{align}%
    &\calB_{L}\subset \calB_{L-1}\subset\ldots\subset \calB_0\eqq \Real^m, \label{Ass_Psi:1}
    \\
    &\calB_{l}=\{ \bu\in \calB_{l-1}\st \Psi_{l}(\bu)\ge 0\}, \qquad \qquad\forall l\in\calL^*,\label{Ass_Psi:2}
    \\
    &\calB_{L} \subset \calB,
    \label{Ass_Psi:3}\\
    &\bu_L, \bu_R \in \calB_l,\quad\text{and}\quad
    \overline{\bu}_{LR}(\lambda_{LR}^{\sharp})\in \calB_l,\quad  \forall l\in\calL.
    \label{Ass_Psi_4}
  \end{align}
\end{subequations}
Notice in passing that all the subsets $\{\calB_l\}_{l\in\calL}$ are
convex since $\calB_0=\Real^m$ and $\calB_{l} = L_0(\Psi_l)$ for all
$l\in\calL^*$. These sets are also closed as the functional
$\{\Psi_l\}_{l\in\calL}$ are continuous. As $\calB_l$ is convex for
all $l\in\calL$, the assumption
\eqref{Ass_Psi_4} then implies that
$\overline{\bu}_{LR}(\lambda)\in \calB_l$ for all
$\lambda\in[\lambda_{LR}^{\sharp}, \infty)$ and
all $l\in\calL$. (The assumption
\eqref{Ass_Psi_4} is reasonable as we already know that
$\overline{\bu}_{LR}(\lambda)\in \calB$ for all
$\lambda\in[\lambda_{LR}^{\sharp}, \infty)$.)

As documented in Appendix A in \cite{Harten_Hyman_1983} (and in Lemma~3.2
in \citep{guermond_popov_second_order_2018}), computing a wave speed that
guarantees a method to be invariant-domain preserving is not enough to
ensure convergence to the entropy solution. Hence, in addition to
invariant-domain properties, we also want to satisfy entropy inequalities.
In order to clarify this objective, we assume to be given a finite set of
entropy pairs for~\eqref{def:hyperbolic_system}, say
$\{(\eta_e,\bq_e)\}_{e\in\calE}$ with $\eta_e:\calB_{L}\to \Real$ and $\bq_e:\calB_{L}\to \Real^d$ for
all $e\in\calE$. Let $\lambda_{LR}^\flat$ be the infimum of the set
$\{\lambda\in [\lambda_\epsilon,\lambda_{LR}^\sharp] \st
\overline{\bu}_{LR}(\lambda) \in \calB_L\}$; that is,
\begin{equation}
  \lambda_{LR}^\flat\eqq \inf \{\lambda\in [\lambda_\epsilon,\lambda_{LR}^\sharp] \st
  \overline{\bu}_{LR}(\lambda) \in \calB_L\}.
\end{equation}
 Note that
$\lambda_{LR}^\flat$ is well defined because the minimizing set is not
empty (it contains $\lambda_{LR}^\sharp$). This infimum is actually
the minimum as
$[\lambda_\epsilon,\lambda_{LR}^\sharp]\ni \lambda\mapsto \bu_{LR}(\lambda)$
is continuous and $\calB_L$ is closed. For every $e\in\calE$, we
introduce the function
$\Phi^e:[\frac{1}{\lambda_{LR}^\sharp},\frac{1}{\lambda_{LR}^\flat}]\to \Real$ defined
by
\begin{equation}%
\Phi_e(t) :=\eta_e\left(\overline\bu_{LR}(\tfrac{1}{t})\right)- \frac{1}{2}(\eta_e(\bu_L)+\eta_e(\bu_R))
      + \frac{t}{2}(\bq_e(\bu_R) - \bq_e(\bu_L))\SCAL\bn, \quad
      \forall e\in\calE.%
      \label{def_Phi_e}
\end{equation}%
We have established in Lemma~\ref{Lem:InvarBar} that
\begin{equation}%
  \Phi_e(1/\lambda_{LR}^{\sharp})\le 0,\quad \forall e\in\calE. \label{Ass_Ulbd_max:2}%
\end{equation}%
Our goal is to find a greedy wave speed
$\lambda^{\textup{grdy}}(\bn,\bu_L,\bu_R)$ as small as possible in
$[\lambda_{LR}^\flat,\lambda_{LR}^\sharp]$ so that
$\overline{\bu}_{LR}(\lambda^{\textup{grdy}}) \in \calB_L$ and
$\Phi_e(1/\lambda^{\textup{grdy}})\le 0$, for all $e\in\calE$.
\begin{lemma} \label{Lem:Phie_is_convex}%
   The function
   $\Phi_e:(\frac{1}{\lambda_{LR}^\sharp},\frac{1}{\lambda_{LR}^\flat})\to
   \Real$ is convex  for all $e\in\calE$.
\end{lemma}
\begin{proof}Let $t_1,t_2\in (\frac{1}{\lambda_{LR}^\sharp},\frac{1}{\lambda_{LR}^\flat})$ and $\theta\in [0,1]$. Then
using that
\begin{align*}
\overline\bu_{LR}(\tfrac{1}{\theta t_1+(1-\theta)t_2})
&= \tfrac{\theta}{2} (\bu_L+\bu_R) \!+\! \tfrac{1-\theta}{2}(\bu_L+\bu_R)
-(\tfrac{\theta}{2} t_1 \!+\! \tfrac{1-\theta}{2} t_2) (\polf(\bu_R)-\polf(\bu_L))\bn\\
& = \theta \overline\bu_{LR}(\tfrac{1}{t_1})
+ (1-\theta)\overline\bu_{LR}(\tfrac{1}{t_2}),
\end{align*}
the assertion follows from the convexity of $\eta_e$.
\end{proof}

\begin{remark}[Notation]
  To be precise the entropy functional defined in \eqref{def_Phi_e} should
  be denoted by $\Phi^e_{LR}$ instead $\Phi_e$ as it depends on the pair
  $\bu_L$, $\bu_R$. Likewise, we should also use $\Psi^l_{LR}$ instead of
  $\Psi_l$. In what follows the index $_{LR}$ reminds us of the dependence
  with respect the pair $\bu_L,$ $\bu_R$ and the unit vector $\bn$. We have
  chosen to use the symbols $\Phi_e$ and $\Psi_l$ instead to simplify the
  notation.
\end{remark}

\begin{remark}[Matryoshka doll structure]
  The \emph{Matryoshka doll structure} introduced in \eqref{Ass_Psi} is
  meant to reflect that the domain of definition of the functionals
  $\Psi_l$ may become smaller and smaller as the index $l$ increases. We
  illustrate this point with the compressible Euler equations with the
  equation of state $p(\bu)\eqq \frac{\gamma-1}{1-b\rho}\rho(e(\bu)-q)
  -\gamma p_\infty$, where $b\ge 0$, $\gamma>1$, $q\in\Real$, and $p_\infty
  \in \Real$, and $e(\bu)\eqq E/\rho-\frac{1}{2} \|\bbm/\rho\|_\ell^2$.
  This equation of state is often called Nobel-Abel stiffened gas equation
  of state in the literature; see \cite{LeMetayer_Saurel_2016}. In this
  case we have: $\Psi_1(\bu) \eqq \rho$, $\calB_1\eqq
  \{\bu=(\rho,\bbm,E)\tr \in\Real^{d+2}\st \rho>0\}$; $\Psi_2(\bu) \eqq 1-
  b\rho$, $\calB_2\eqq \{\bu\in\calB_1\st 1-b\rho>0\}$; $\Psi_3(\bu) \eqq
  \rho(e(\bu) -q) - p_\infty(1-b\rho)$, $\calB_3\eqq  \{\bu\in\calB_2\st
  \rho(e(\bu) -q) - p_\infty(1-b\rho)>0\}$. Notice that the constraint
  $\Psi_3(\bu)>0$ implies that $p(\bu)+p_\infty>0$ which is essential to be
  able to define the specific entropy $\eta(\bu)
  =\log((1/\rho-b)^{\gamma}(p(\bu)+p_\infty))$.
\end{remark}

In practice, we are going to enforce sharper bounds than those shown above
by making all the functionals $\{\Psi_l\}_{l\in\calL}$ and all the sets
$\{\calB_l\}_{l\in\calL}$ depend on the states $\bu_L$ and $\bu_R$ (see
\S\ref{Sec:Scalar_equations} and \S\ref{Sec:p_system}).

%%%%%%%%%%%%%%%%%%%%%%%%%%%%%%%%%%%%%%%%%%%%%%%%%%%%%%%%%%%%%%%%%%%%%%%%%%%%%%%%
\subsection{Algorithm for estimating the greedy wave speed}
\label{Sec:def_greedy_wave_speed}

As mentioned above, the key idea of the paper is to define a greedy
wave speed $\lambda^{\textup{grdy}}(\bn,\bu_L,\bu_R)$ in
$[\lambda_{LR}^\flat,\lambda_{LR}^\sharp]$ so that
$\overline{\bu}_{LR}(\lambda^{\textup{grdy}}) \in \calB_L$ and
$\Phi_e(1/\lambda^{\textup{grdy}})\le 0$, for all $e\in\calE$. We
now present an algorithm that carries out this program (see
Algorithm~\ref{Greedy_algorithm}).

One starts by setting
$\lambda_0(\bn,\bu_L,\bu_R)\eqq\lambda_\epsilon$. Then one traverses
the index set $\calL^*$ in increasing order, and for each index $l$ in
$\calL^*$ one computes the wave speed $\lambda_l(\bn,\bu_L,\bu_R)$
recursively defined by
\begin{align}
  &\lambda_l:=
  \min\{ \lambda\in [\lambda_{l-1},\lambda_{LR}^{\sharp}]\st
  \Psi_l(\overline\bu_{LR}(\lambda))\ge 0\}.
  \label{Abstract_Psil}
\end{align}
Next, one (indiscriminately) traverses the index set $\calE$ and
computes the wave speed $\lambda_e(\bn,\bu_L,\bu_R)$ defined by
\begin{align}
  &\lambda_e:=
  \min\{\lambda\in [\lambda_L,\lambda_{LR}^{\sharp}]\st
  \Phi_e(\lambda^{-1})\le 0\}.
  \label{Abstract_etae}
\end{align}
One finally defines the greedy wave speed $\lambda\upgrdy(\bn,\bu_L,\bu_R)$ as follows:
\begin{equation}
 \lambda\upgrdy(\bn,\bu_L,\bu_R) :=
\max_{e\in \calE} \lambda_e. \label{def_of_greedy_viscosity}
\end{equation}
Techniques to compute the wave speed defined in \eqref{Abstract_Psil} and
\eqref{Abstract_etae} are explained in \S\ref{Sec:Maximum_principle} and
\S\ref{Sec:entropy_ineq_lambda_scalar} for nonlinear scalar equations and
in \S\ref{IDP_p_system} and \S\ref{Sec:p_system_entropy_lambda} for the
p-system.

 \begin{lemma}\label{Lem:def_lambdas}
   Assume that \eqref{Ass_Psi} hold true. Then,
\begin{enumerate}[\rm(i)]
\item \label{Item1:Lem:def_lambdas} $\lambda_l$ is well defined and
  $\lambda_\epsilon\le \lambda_l \le \lambda_{LR}^\sharp$ for all $l\in\calL$.
We have
  $\overline\bu_{LR}(\lambda) \in \calB_{l}$
  for all $\lambda\in [\lambda_{l},\lambda_{LR}^\sharp]$ and all $l\in\calL$.
\item \label{Item2:Lem:def_lambdas}
  $\lambda_e$ is well defined.
 We have $\Phi_e(\frac{1}{\lambda}) \le 0$ for all $\lambda \in [\lambda_e,\lambda_{LR}^\sharp)]$
and all $e\in\calE$.
\end{enumerate}
\end{lemma}
\begin{proof} Recall that
  $\lambda_{LR}^{\sharp}\eqq\max(\lambda_\epsilon,\lambda_{\max})$.\\
  \eqref{Item1:Lem:def_lambdas} We proceed by induction over
  $l\in\calL$.  The wave speed $\lambda_\epsilon$ is well defined (see
  \eqref{def_lambda_0}) and
  $\lambda_\epsilon \in
  [\lambda_\epsilon,\lambda_{LR}^\sharp]$. Moreover,
  $\overline{\bu}_{LR}(\lambda)\in \calB_0\eqq\Real^{m}$ for all
  $\lambda\in [\lambda_\epsilon,\lambda_{LR}^\sharp]$.  Hence, the
  induction assumption \eqref{Item1:Lem:def_lambdas} holds for $l=0$
  since $\lambda_0\eqq \lambda_\epsilon$.  Now let $l\in\calL^*$ and
  let us prove that \eqref{Item1:Lem:def_lambdas} holds. The induction
  assumption implies that the set
  $[\lambda_{l-1},\lambda_{LR}^\sharp]$ is not empty (because
  $\lambda_{l-1}\le \lambda_{LR}^\sharp$), and
  $\overline{\bu}_{LR}(\lambda) \in \calB_{l-1}$ for all
  $\lambda\in [\lambda_{l-1},\lambda_{LR}^\sharp]$. This means in
  particular that $\Psi_{l}(\overline{\bu}_{LR}(\lambda))$ is well
  defined for all $\lambda\in [\lambda_{l-1},\lambda_{LR}^\sharp]$.
  Moreover, we have
  $\overline{\bu}_{LR}(\lambda_{LR}^\sharp)\in\calB_{l}$ owing to the
  assumption \eqref{Ass_Psi_4}. Hence the set
  $\{\lambda\in [\lambda_{l-1},\lambda_{LR}^\sharp]\st
  \Psi_{l}(\overline{\bu}_{LR} (\lambda))\ge 0\}$ is not empty. This
  set has a minimum since $\Psi_l$ is continuous, the mapping
  $[\lambda_{l-1},\lambda_{LR}^\sharp] \ni \lambda\mapsto
  \overline{\bu}_{LR} (\lambda)\in \calB_{l-1}$ is continuous, and
  $[\lambda_{l-1},\lambda_{LR}^\sharp]$ is compact. Hence
  $\lambda_{l}$ is well defined and
  $\lambda_\epsilon\le \lambda_{l-1}\le
  \lambda_l\le\lambda_{LR}^\sharp$ (by definition). Let us now prove
  that $\overline{\bu}_{LR}(\lambda) \in \calB_{l}$ for all
  $ \lambda\in [\lambda_{l},\lambda_{LR}^\sharp]$. We first observe
  that
  $\overline{\bu}_{LR}(\lambda) = \theta\,
  \overline{\bu}_{LR}(\lambda_{l}) + (1-\theta)
  \overline{\bu}_{LR}(\lambda_{LR}^\sharp)$ with
  $\theta\eqq \frac{(\lambda_{LR}^\sharp -\lambda)\,\lambda_{l}}
  {(\lambda_{LR}^\sharp -\lambda_{l})\,\lambda}$; hence, the set
  $\{\overline{\bu}_{LR}(\lambda)\st \lambda\in
  [\lambda_{l},\lambda_{LR}^\sharp]\}$ is a line segment in
  $\Real^m$. But both $\overline{\bu}_{LR}(\lambda_{l})$ and
  $\overline{\bu}_{LR}(\lambda_{LR}^\sharp) $ are members of
  $\{ \bu \in \calB_{l-1}\st \Psi_{l}(\bu)\ge 0\}= \calB_{l}$. Since
  $\calB_{l}$ is convex, we conclude that the entire line segment
  $\{\overline{\bu}_{LR}(\lambda)\st \lambda\in
  [\lambda_{l},\lambda_{LR}^\sharp]\}$ is in $\calB_{l}$. This proves
  that the induction
  assumption holds true for $l$. \\
  \eqref{Item2:Lem:def_lambdas} The argument in
  \eqref{Item1:Lem:def_lambdas} proves that
  $\overline{\bu}_{LR}(\lambda)\in \calB_{L}$ for all
  $\lambda\in [\lambda_L,\lambda_{LR}^\sharp]$. As the domain of
  $\eta_e$ and $\bq_e$ is $\calB_{L}$, this argument proves that
  $\Phi_e(\frac{1}{\lambda})$ is well defined for all
  $\lambda\in [\lambda_L,\lambda_{LR}^\sharp]$ and all
  $e\in\calE$. The continuity of $\Phi_e$ implies that $\lambda_e$ is
  well defined as well. From the convexity of $\Phi_e$ established in
  Lemma~\ref{Lem:Phie_is_convex} it follows that
  $\Phi_e(\frac{1}{\lambda}) \le 0$ for all
  $\lambda \in [\lambda_e,\lambda_{LR}^\sharp]$ since
  $\Phi_e(\frac{1}{\lambda_e}) \le 0$ and
  $\Phi_e(\frac{1}{\lambda_{LR}^\sharp}) \le 0$,
  see~\eqref{Ass_Ulbd_max:2}.
 \end{proof}

\begin{algorithm}[t]
  \renewcommand{\algorithmicrequire}{\textbf{Input:}}
  \renewcommand{\algorithmicensure}{\textbf{Output:}}
  \caption{Greedy wave speed}
  \label{Greedy_algorithm}
  \begin{algorithmic}[1]
    \Require $\bn$, $\bu_L$, $\bu_R$
    \Ensure $\lambda\upgrdy(\bn,\bu_L,\bu_R)$
    \State Compute $\lambda_{\max}(\bn,\bu_L,\bu_R)$,
    $\lambda_0(\bn,\bu_L,\bu_R)$ and $\lambda_{LR}^\sharp$
    \For{$l=1 \textbf{ to } L$}
    \State Define $\Psi_l^{LR}$ and compute $\lambda_l(\bn,\bu_L,\bu_R)$;
    see \eqref{Abstract_Psil}
    \EndFor
    \For{$e\in \calE$}
    \State Define $(\eta_e^{LR},\bq_e^{LR})$ and compute $\lambda_e(\bn,\bu_L,\bu_R)$; see \eqref{Abstract_etae}
    \EndFor
    \State  Compute $\lambda\upgrdy(\bn,\bu_L,\bu_R)$; see~\eqref{def_of_greedy_viscosity}
  \end{algorithmic}
\end{algorithm}

%%%%%%%%%%%%%%%%%%%%%%%%%%%%%%%%%%%%%%%%%%%%%%%%%%%%%%%%%%%%%%%%%%%%%%%%%%%%%%%%
\subsection{Greedy viscosity}
We are now in a position to state the main results of
\S\ref{Sec:General_strategy}. Using the same notation as in
\S\ref{Sec:aux_bar_states}, let $i\in\calV$ and $j\in\calI(i)$. With the
greedy wave speed $\lambda^{\textup{grdy}}(\bn_{ij},\bsfU_i^n,\bsfU_j^n)$
defined in \eqref{def_of_greedy_viscosity}, we define the greedy viscosity
for the pair $(i,j)$ at the time $t^n$ as follows:
\begin{equation}
  d_{ij}\upgrdyn = \max(\lambda\upgrdy(\bn_{ij},\bsfU_i^n,\bsfU_j^n)
  \|\bc_{ij}\|_{\ell^2},
  \lambda\upgrdy(\bn_{ji},\bsfU_j^n,\bsfU_i^n)\|\bc_{ji}\|_{\ell^2}).
  \label{def_dij_greedy}
\end{equation}
Note that if $\lambda_{\max}(\bn_{ij},\bsfU_i^n,\bsfU_j^n)\ge \lambda_\epsilon$
(which is almost always the case), then
\begin{equation}
  d_{ij}\upGMSn \ge d_{ij}\upgrdyn.
\end{equation}
The main result of the paper and the reason we have introduced the greedy
wave speed is the following.
\begin{theorem}[IDP Greedy viscosity]
  \label{Thm:UL_is_invariant_greedy} Let $\calB$ be an invariant
  domain for \eqref{def:hyperbolic_system}.
  Let $n\ge 0$, $i\in\calV$. For all $j\in\calI(i)^*$, let
  $\{\calB_l^{ij}\}_{l\in\calL}$ be a finite collection of convex sets,
  and let $\{\Psi_l^{ij}:\calB_l^{ij}\to
  \Real\}_{l\in\calL^*}$ be a collection of continuous quasiconcave
  functionals. Let $\{(\eta_e^i,\bq_e^i)\}_{e\in\calE^i}$ be a
  finite set of entropy pairs for~\eqref{def:hyperbolic_system}.
  (We use a superscript $i$ on the entropy pairs to allow for the
  possibility to choose a different set of entropies for each index
  $i\in\calV$.)
  Let $\{d_{ij}\upgrdyn\}_{j\in\calI(i)^*}$ be the greedy
  graph viscosity defined by \eqref{def_dij_greedy} and let
  $\{\bsfU_i\upnp\}_{i\in\calV}$ be the update defined in
  \eqref{low_order_scheme} with the choice $d_{ij}^n\eqq d_{ij}\upgrdyn$.
  Assume the following:
  \begin{enumerate}[\rm(i)]
  \item
    $\{\calB_l^{ij}\}_{l\in\calL}$ and
    $\{\Psi_l^{ij}\}_{l\in\calL^*}$ satisfy the assumptions
    in~\eqref{Ass_Psi} for all $j\in\calI(i)^*$;
  \item
    $\dt$ is small enough so that
    $1-2\frac{\dt}{m_i}\sum_{j\in\calI(i)^*}d_{ij}^n> 0$.
  \end{enumerate}
  Then the update $\{\bsfU_i\upnp\}_{i\in\calV}$ satisfies the following
  properties:
  \begin{align}
    &\bsfU_i\upnp\in\textup{conv}\,\Big(\bigcup_{j\in\calI(i)^*}\calB_{L}^{ij}\Big),
    \quad\text{hence}\quad\bsfU_i\upnp \in \calB,
    \label{eq1:Thm:UL_is_invariant_greedy}
    \\
    &\frac{m_i}{\dt} (\eta_e^i(\bsfU_i\upnp) - \eta_e^i(\bsfU_i^n))
    + \!\!\sum_{j\in\calI(i)}\!\! \bc_{ij}\SCAL \bq_e^i(\bsfU_j^n) -\!\!\!\!\!
    \sum_{j\in \calI(i)^*} \!\!\! d_{ij}^n
    (\eta_e^i(\bsfU_{j}^{n})  - \eta_e^i(\bsfU_{j}^{n})) \le 0.
    \label{eq2:Thm:UL_is_invariant_greedy}
  \end{align}
\end{theorem}

\begin{proof}
  We first recall that~\eqref{low_order_scheme} can be rewritten as
  follows:
  \begin{align}
    \bsfU_i\upnp =
    \bigg(1 -\!\! \sum_{j\in\calI(i)^*} \frac{2 \dt
    d_{ij}\upn}{m_i} \bigg)
    \bsfU_{i}^{n}
      + \!\! \sum_{j\in\calI(i)^*} \frac{2 \dt d_{ij}\upn}{m_i}
    \overline{\bsfU}_{ij}\upn, \label{convex_bombination}
  \end{align}
  with the notation $\overline{\bsfU}_{ij}\upn \eqq
  \overline{\bu}_{ij}(\frac{d_{ij}^n}{\|\bc_{ij}\|_{\ell^2}})$. Setting
  $\lambda_{ij}:= \frac{d_{ij}^n}{\|\bc_{ij}\|_{\ell^2}}$ for all
  $j\in\calI(i)^*$, we have $\overline{\bsfU}_{ij}\upn =
  \overline{\bu}_{ij}(\lambda_{ij})$. As the assumptions in
  \eqref{Ass_Psi} hold and $\lambda\upgrdy(\bn_{ij},\bsfU_i^n,\bsfU_j^n)$
  is defined by
  \eqref{Abstract_Psil}-\eqref{Abstract_etae}-\eqref{def_of_greedy_viscosity}
  for all $j\in\calI(i)^*$, we can apply
  Lemma~\ref{Lem:def_lambdas}.
  Then combining \eqref{def_of_greedy_viscosity}
  with the identity
  $\lambda_{ij} \|\bc_{ij}\|_{\ell^2} = d_{ij}^n = d_{ij}\upgrdyn$
  implies
  $\lambda_{ij} \ge \lambda_{ij}\upgrdy(\bn_{ij},\bsfU_i^n,\bsfU_j^n)$,
  and invoking Lemma~\ref{Lem:def_lambdas}\eqref{Item1:Lem:def_lambdas},
  \eqref{Ass_Psi:1} and \eqref{Ass_Psi:3}, we infer that
  \begin{align*}
    \overline{\bsfU}^n_{ij}\in \bigcap_{l\in\calL} \calB_l^{ij}
    = \calB_{L}^{ij} \subset \calB.
  \end{align*}
  Since we assumed that
  $1-2\frac{\dt}{m_i}\sum_{j\in\calI(i)^*}d_{ij}^n> 0$, the
  right-hand side in \eqref{convex_bombination} is a convex combination of
  the states $\{\bsfU_i^n\}\cup
  \{\overline{\bsfU}_{ij}\upn\}_{j\in\calI(i)^*}$ which all
  lie in
  the convex hull $\textup{conv} \Big(\bigcup_{j\in\calI(i)^*}\calB_{L}^{ij}\Big)$,
  and it follows that $\bsfU_i\upnp \in  \calB$. Let us now establish the
  entropy inequality \eqref{eq2:Thm:UL_is_invariant_greedy}. From the
  convexity of $\eta_e^i$ and \eqref{convex_bombination} we obtain
  \begin{align*}
    \eta_e^i(\bsfU_i\upnp) \le
    \bigg(1 -\!\! \sum_{j\in\calI(i)^*} \frac{2 \dt
    d_{ij}\upn}{m_i} \bigg) \eta_e^i(\bsfU_{i}^{n})
    + \!\! \sum_{j\in\calI(i)^*} \frac{2 \dt d_{ij}\upn}{m_i}
    \eta_e^i(\overline{\bsfU}_{ij}\upn).
  \end{align*}
  Using Lemma~\ref{Lem:def_lambdas}\eqref{Item2:Lem:def_lambdas} and
  recalling that $\overline{\bsfU}_{ij}\upn =
  \overline{\bu}_{ij}(\lambda_{ij})$, we infer that
  $\Phi_e(\frac{1}{\lambda_{ij}})\le 0$, \ie
  \begin{align*}
    2 d_{ij}\upn\eta_e^i(\overline{\bsfU}_{ij}\upn) \le
    d_{ij}\upn(\eta_e^i(\bsfU_i^n) + \eta_e^i(\bsfU_j^n)) -
    (\bq_e^i(\bsfU_j^n) - \bq_e^i(\bsfU_i^n))\SCAL\bc_{ij}.
  \end{align*}
  Inserting this inequality in the previous inequality and
  using~\eqref{cijprop} gives~\eqref{eq2:Thm:UL_is_invariant_greedy}.
\end{proof}
\begin{remark}
  More generally, Theorem~\ref{Thm:UL_is_invariant_greedy} holds true for
  any choice $\{d_{ij}^n\}_{j\in\calI(i)^*}$ of graph
  viscosity provided that $d_{ij}^n\ge d_{ij}\upgrdyn$ for all
  $j\in\calI(i)^*$, $i\in\calV$.
\end{remark}

The result stated in  Theorem~\ref{Thm:UL_is_invariant_greedy} can be
slightly refined by assuming a little more structure on the sets
$\{\calB^{ij}_l\}_{i\in\calI(i)^*}$ for all $l\in\calL$.
\begin{corollary}[Localization]
  \label{Cor:UL_is_invariant_greedy_refinement}
  Let the assumptions of Theorem~\ref{Thm:UL_is_invariant_greedy} hold.
  Assume also that the following holds true for all $l\in\calL^*$: There
  exists $i(l)\in\calI(i)^*$ so that $\calB^{ij}_l \subset \calB^{i
  i(l)}_l$ for all $j\in\calI(i)^*$. Then the update  given by
  \eqref{low_order_scheme} satisfies the following local properties:
  \begin{align}
    \Psi_l^{i i(l)}(\bsfU_i\upnp)\ge 0,\quad \forall l\in\calL^*.
  \end{align}
\end{corollary}
\begin{proof}
  The assumption together with with property
  \eqref{eq1:Thm:UL_is_invariant_greedy} from
  Theorem~\ref{Thm:UL_is_invariant_greedy} implies that
  \begin{align*}
    \bsfU_i\upnp\in\textup{conv}\,\Big(\bigcup_{j\in\calI(i)^*}
    \calB_{L}^{ij}\Big) \subset \Big(\bigcup_{j\in\calI(i)^*}
    \calB_{l}^{ij}\Big) = \calB^{i i(l)}_l.
  \end{align*}
  Hence, the statement is an immediate consequence of the definition of the
  set $\calB^{i i(l)}_l$ given in \eqref{Ass_Psi:2}.
\end{proof}

\begin{remark}
  Computing the maximal wave speed
  $\lambda_{\max}(\bn,\bu_L,\bu_R)$ for general hyperbolic systems
  typically requires solving a nonlinear scalar fixed point
  problem. Computing an upper bound on
  $\lambda_{\max}(\bn,\bu_L,\bu_R)$ is somewhat simpler as it
  requires to use iterative techniques that converge from
  above. Very accurate upper bounds are usually obtained in two to
  three iterations. The time spent to this task is in general
  negligible. For instance, the reader is referred to
  \citep{Guermond_Popov_Fast_Riemann_2016} where guaranteed upper
  bounds on $\lambda_{\max}(\bn,\bu_L,\bu_R)$ are given for the
  Euler equations with the co-volume equation of state (the source
  code for this method is available in the appendix of
  \citep{Guermond_Popov_Fast_Riemann_2016} and a source code
  computing $\lambda_{\max}(\bn,\bu_L,\bu_R)$ for a general equation
  of state is available at
  \cite{Clayton_Guermond_Popov_2021}). Computing
  $\lambda\upgrdy(\bn,\bu_L,\bu_R)$ or an upper bound thereof is
  similar to estimating upper bounds for
  $\lambda_{\max}(\bn,\bu_L,\bu_R)$. This can be easily done by using
  iterative techniques converging from above.
\end{remark}
%
%%%%%%%%%%%%%%%%%%%%%%%%%%%%%%%%%%%%%%%%%%%%%%%%%%%%%%%%%%%%%%%%%%%%%%%%%%%%%%%%
%%%%%%%%%%%%%%%%%%%%%%%%%%%%%%%%%%%%%%%%%%%%%%%%%%%%%%%%%%%%%%%%%%%%%%%%%%%%%%%%
%%%%%%%%%%%%%%%%%%%%%%%%%%%%%%%%%%%%%%%%%%%%%%%%%%%%%%%%%%%%%%%%%%%%%%%%%%%%%%%%

\section{Scalar conservation equations}
\label{Sec:Scalar_equations}
In this section we specialize the proposed definitions
\eqref{def_of_greedy_viscosity}-\eqref{def_dij_greedy} on scalar
conservation equations. Instead of using the notation
$\polf$ and $\bu$, we now denote the flux by $\bef$ and the
dependent variable by $u$.

%%%%%%%%%%%%%%%%%%%%%%%%%%%%%%%%%%%%%%%%%%%%%%%%%%%%%%%%%%%%%%%%%%%%%%%%%%%%%%%%
\subsection{Maximum principle}\label{Sec:Maximum_principle}
In the scalar case, the only invariant-domain property there is reduces to
enforcing the maximum principle. We start by estimating a wave speed that
does exactly that by following the algorithm \eqref{Abstract_Psil} described in
\S\ref{Sec:def_greedy_wave_speed}. We take care of the entropy
inequalities~\eqref{Abstract_etae} in
\S\ref{Sec:entropy_ineq_lambda_scalar}

Let $ u_L, u_R\in \calA$ and let $\bn$ be a unit vector in
$\Real^d$. (Computing $\lambda_{\max}(\bn, u_L, u_R)$ is a standard exercise; see
\eg \cite[Lem.\,3.1]{Dafermos_1972}, \cite[\S\,2.2]{Holden_Risebro_2015},
\cite[Thm.\,1]{Osher_1983}.) We introduce two concave functionals to take
care of the local minimum and maximum principle:
\begin{subequations}
  \begin{align}
    u_{LR}^{\min}&\eqq \min( u_L, u_R),
    && u_{LR}^{\max}\eqq \max( u_L, u_R),
    \\
    \Psi_1(u)&\eqq u - u_{LR}^{\min},
    && \Psi_2(u)\eqq u_{LR}^{\max}-u.
  \end{align}
\end{subequations}
Accordingly, we set $\calB_0\eqq \Real$, $\calB_1\eqq \{ u\in\calB_0\st
\Psi_1(u)\ge 0\}$, $\calB_2\eqq \{ u\in\calB_1\st \Psi_2(u)\ge 0\}$.
\begin{lemma}
  \label{Lem:Scalar_minmax_Roe_speed}
  Let
  \begin{equation}
    \label{def_lambda_12}
    \lambda_{12}(\bn, u_L, u_R):=
    \begin{cases}
      \frac{|(\bef( u_R) - \bef( u_L))\SCAL \bn|}{| u_R- u_L|} & \text{if $ u_R\ne u_L$}
      \\
      \max(|\bef'( u_R)\SCAL \bn|,|\bef'( u_L)\SCAL \bn|) & \text{if $ u_R= u_L$}.
    \end{cases}
  \end{equation}
  Then $\Psi_1(\overline{u}_{LR}(\lambda))\ge 0$ and
  $\Psi_2(\overline{u}_{LR}(\lambda))\ge 0$ for all $\lambda\ge
  \max(\lambda_{12},\lambda_\epsilon)$. (This also means
  $\overline{u}_{LR}(\lambda)\in [u_{LR}^{\min},u_{LR}^{\max}]$ for all
  $\lambda\ge \max(\lambda_{12},\lambda_\epsilon)$.)
\end{lemma}
\begin{proof}
  Let $a:=\frac{1}{2}( u_L +  u_R)$,
  $b:=(\bef(u_R) - \bef(u_L))\SCAL \frac{\bn}{2}$. Note that
  $a\in [u_{LR}^{\min},u_{LR}^{\max}]$ and recall that
  $\overline{u}_{LR}(\lambda) \eqq a- b\lambda^{-1}$, for all $\lambda\ge \lambda_\epsilon>0$. We want to
  estimate the smallest value of $\lambda$ in $[\lambda_\epsilon,\lambda_{LR}^\sharp]$ so that
  $\Psi_1(\overline{u}_{LR}(\lambda)) \ge 0$ and $\Psi_2(\overline{u}_{LR}(\lambda)) \ge 0$.
  That is, we want $\lambda$ to be such that
  \begin{align*}
    -\frac12 (u_{LR}^{\max}-u_{LR}^{\min}) =
    a-u_{LR}^{\max} \le b \lambda^{-1} \le  a-u_{LR}^{\min} =
    \frac12 (u_{LR}^{\max}-u_{LR}^{\min}).
  \end{align*}
  This holds true if and only if $|b|\lambda^{-1} \le \frac12
  (u_{LR}^{\max}-u_{LR}^{\min})$. If $u_{LR}^{\max}-u_{LR}^{\min}\ne 0$,
  the smallest possible value of $\lambda$ making this inequality to hold
  is $\lambda=\frac{|(\bef( u_R) - \bef(
  u_L))\SCAL\bn|}{u_{LR}^{\max}-u_{LR}^{\min}}$.
  If $u_{LR}^{\max}-u_{LR}^{\min}= 0$, every value of $\lambda$ is
  admissible, but the only value of $\lambda$ that is stable under
  perturbation of the two states is $\lambda=|\bef'( u_R)\SCAL \bn|$
  if $\bef$ is of class $C^1$, and $\max(|\bef'( u_R)\SCAL
  \bn|,|\bef'( u_L)\SCAL \bn|)$ otherwise.
\end{proof}

We note that the wave speed identified in
Lemma~\ref{Lem:Scalar_minmax_Roe_speed}, $\frac{|(\bef( u_R) - \bef(
u_L))\bn|}{| u_R- u_L|}$, is the average speed, sometimes called Roe's
average in the computational fluid dynamics literature.
As the final wave speed defining the artificial viscosity is
eventually larger than or equal to this quantity, Lemma~2 from
\cite{Harten_JCP_1983} implies that the scheme is \emph{total variation non
increasing} in one space dimension on the three point stencil (the wave
speed $\lambda_{12}$ also satisfies the necessary and sufficient condition
formulated in \cite[Cor.~2.3]{Tadmor_Math_Com_1984}). It is well known
that in the presence of sonic points this wave speed is not large enough to
ensure that the approximation defined in \eqref{low_order_scheme} converges
to the entropy solution (see, \eg \cite[App.~ A]{Harten_Hyman_1983} or
\citep[Lem.\,3.2]{Guermond_Popov_SINUM_2017} for a simple proof). This
problem is addressed in the next section by augmenting the wave speed so as
to make sure that some entropy inequalities are locally satisfied, \ie
\eqref{Abstract_etae} is satisfied.

%%%%%%%%%%%%%%%%%%%%%%%%%%%%%%%%%%%%%%%%%%%%%%%%%%%%%%%%%%%%%%%%%%%%%%%%%%%%%%%%
\subsection{Entropy inequality}
\label{Sec:entropy_ineq_lambda_scalar}
Now, following algorithm \eqref{Abstract_etae} described in
\S\ref{Sec:def_greedy_wave_speed}, we further look for a wave speed,
possibly larger than $\lambda_{12}$, so as to satisfy some entropy
inequalities.
\begin{lemma}
  \label{Lem:Scalar_entropy_speed}
  Let $k \in \Real$. Let $\eta_k(u) := |u - k|$ be the Kr\v{u}zkov
  entropy associated with $k$ and $\bq_k(u) := \sign(u -k)
  (\bef(u) -\bef(k))$ be the corresponding entropy flux. Let
  \begin{align*}
    &a_k:= u_L+ u_R - 2k, & & b:=(\bef( u_R)-\bef( u_L))\SCAL \bn,
    \\
    &c_k:=\eta_k( u_L) + \eta_k( u_R), & &d_k:=(\bq_k( u_R)-\bq_k( u_L))\SCAL \bn .
  \end{align*}
  (Observe that $|a_k|=c_k$ if and only if
  $k\not\in(u_{LR}^{\min},u_{LR}^{\max})$.) Let
  $\lambda_{12}(\bn, u_L, u_R)$ be defined as in
  Lemma~\ref{Lem:Scalar_minmax_Roe_speed}, and let
  \begin{equation}
    \label{Lambda_Scalar_Entrop_Ineq}
    \lambda(k,\bn, u_L, u_R)\eqq\begin{cases}
    \lambda_{12}(\bn, u_L, u_R)& \text{if $k\not\in (u_{LR}^{\min},u_{LR}^{\max})$}
    \\
    \max\left( \frac{d_k+b}{c_k+a_k},\frac{d_k-b}{c_k-a_k},
    \lambda_{12}(\bn, u_L, u_R) \right) & \text{otherwise}.
    \end{cases}\hspace{-.2cm}
  \end{equation}
  Let $\Phi_k(\frac{1}{\lambda})\eqq \eta_k(\overline{u}_{LR}(\lambda))
  - \frac12(\eta_k( u_L) + \eta_k( u_R)) +
  \frac{1}{2\lambda}(\bq_k( u_R)-\bq_k( u_L))\SCAL \bn$. Then, for
  every $\lambda\ge \max(\lambda(k,\bn, u_L, u_R),\lambda_\epsilon)$ we
  have $\Phi_k(\frac{1}{\lambda})\le 0$ .
\end{lemma}
\begin{proof}
  (1) Assume first that $k\not\in (u_{LR}^{\min},u_{LR}^{\max})$, \ie
  $c_k = |a_k|$. The assumption $\lambda\ge
  \max(\lambda_{12},\lambda_\epsilon)$ implies that
  $\overline u_{LR}(\lambda) \in [u_{LR}^{\min}, u_{LR}^{\max}]$.
  Hence, $\sign(\overline u_{LR}(\lambda)-k) = \sign(\tfrac12 ( u_R+ u_L)
  -k)$. As a result, we have
  \begin{align*}
    \eta_k(\overline{u}_{LR}(\lambda)) & = \sign(\overline u_{LR}(\lambda) -k)
    (\overline u_{LR}(\lambda) -k)
    \\
     & = \sign(\tfrac12 ( u_R+ u_L) -k) \left(\tfrac12 ( u_R+ u_L)
    -\tfrac{1}{2 \lambda} (\bef( u_R) -\bef( u_L))\SCAL \bn - k\right).
  \end{align*}
  On the other hand, using that $\eta_k( u_R)=\sign(\tfrac12
  ( u_R+ u_L) -k) ( u_R -k)$, $\bq_k( u_R) =\sign(\tfrac12
  ( u_R+ u_L) -k)(\bef( u_R) - \bef(k))$, and the corresponding
  identities for $\eta_k( u_L)$ and $\bq_k( u_L)$, we deduce that
  \begin{align*}
    \tfrac12 \eta_k( u_L) &+ \tfrac12 \eta_k( u_R)
    - \tfrac{1}{2\lambda} (\bq_k( u_R)-\bq_k( u_L))\SCAL \bn  \\
    & = \sign(\tfrac12 ( u_R+ u_L) -k) \left(\tfrac12 ( u_R+ u_L)
    -\tfrac{1}{2\lambda}(\bef( u_R) -\bef( u_L))\SCAL \bn - k\right).
  \end{align*}
  Hence, we conclude that $\eta_k(\overline u_{LR}(\lambda))=
  \eta_k( u_L) + \eta_k( u_R) - \lambda^{-1}
  (\bq_k( u_R)-\bq_k( u_L))\SCAL \bn $ for all $k\not\in
  (u_{LR}^{\min},u_{LR}^{\max})$.

  (2) Let us now assume that $k \in (u_{LR}^{\min},u_{LR}^{\max})$.
  Then we have that $c_k-|a_k|\ge 2 \min(\eta_k( u_L),\eta_k( u_R))>0$.
  Hence definition \eqref{Lambda_Scalar_Entrop_Ineq} makes sense. Using the
  definitions for $a_k$, $b$, $c_k$, and $d_k$, we have $2
  \eta_k(\overline{u}_{LR}(\lambda)) = |a_k - \lambda^{-1} b|$. Then we want
  to find the smallest value of $\lambda$ that guarantees that
  \begin{align*}
    |a_k - \lambda^{-1}b| \le c_k - \lambda^{-1} d_k.
  \end{align*}
  The above inequality is equivalent to
  \begin{align*}
    \lambda^{-1}(d_k-b) \le c_k-a_k, \quad\text{and}\quad
    \lambda^{-1}(b+d_k)  \le c_k + a_k.
  \end{align*}
  Using that $|a_k|<c_k$, we infer that
  \begin{align*}
    \lambda \ge \frac{d_k-b}{c_k-a_k}; \qquad \lambda \ge \frac{d_k+b}{c_k+a_k}.
  \end{align*}
  The assertion follows readily.
\end{proof}

%%%%%%%%%%%%%%%%%%%%%%%%%%%%%%%%%%%%%%%%%%%%%%%%%%%%%%%%%%%%%%%%%%%%%%%%%%%%%%%%
\subsection{Summary}
The following result summarizes what is proposed above.
In particular, it shows how the Kr\v{u}zkov entropies should be chosen.
\begin{theorem}
  \label{Thm:UL_is_invariant_modified}
  Let $n\ge 0$, $i\in\calV$,
  $\sfU_i^{\min,n} \eqq\min_{j\in\calI(i)}\sfU_j^n$,
  $\sfU_i^{\max,n} \eqq\max_{j\in\calI(i)}\sfU_j^n$.
  Let $k_i$ be any real number in the range
  $(\sfU_i^{\min,n},\sfU_i^{\max,n})$. Let $(\eta_{k_i},\bq_{k_i})$ be the
  associated Kr\v{u}zkov entropy pair. For all $j\in\calI(i)^*$, let
  $\lambda_{ij}\upgrdyn\eqq \max(\lambda_\epsilon,
  \lambda(k_i,\bn_{ij},\sfU_i^m,\sfU_j^n))$ and
  \begin{equation}
   d_{ij}\upgrdyn :=
    \max(\lambda_{ij}\upgrdyn\|\bc_{ij}\|_{\ell^2},
    \lambda_{ji}\upgrdyn\|\bc_{ji}\|_{\ell^2}).
  \end{equation}
  Let $\sfU_i\upnp$ be given by \eqref{low_order_scheme} with the
  viscosity $d_{ij}^n=d_{ij}\upgrdyn$ defined above. Assume that $1-2\frac{\dt}{m_i}\sum_{j\in\calI(i)^*}d_{ij}^n\ge 0$. Then
  \begin{align}
    &\sfU_i\upnp \in [\sfU_i^{\min,n},\sfU_i^{\max,n}]
    \label{eq1:Thm:UL_is_invariant_modified}
    \\
    &\frac{m_i}{\dt} (\eta_{k_i}(\sfU_i\upnp) - \eta_{k_i}(\sfU_i^n))
    + \sum_{j\in\calI(i)} \bc_{ij}\SCAL \bq_{k_i}(\sfU_j^n)
    \label{eq2:Thm:UL_is_invariant_modified}
    \\
    &\hspace{3cm}- \sum_{j\in \calI(i)^*} \!\!\! d_{ij}^n
    (\eta_{k_i}(\sfU_{j}^{n})  - \eta_{k_i}(\sfU_{j}^{n})) \le 0.
    \nonumber
  \end{align}
\end{theorem}

\begin{proof}
  This is just a reformulation of
  Theorem~\ref{Thm:UL_is_invariant_greedy} .
\end{proof}

\begin{remark}[Entropy choice]
  It is essential that $k_i$ be chosen in
  $(\sfU_i^{\min,n},\sfU_i^{\max,n})$; otherwise, we have
  $\lambda_{ij}\upgrdyn\eqq \max(\lambda_\epsilon,
  \lambda_{12}(\bn_{ij},\sfU_i^m,\sfU_j^n)$, and
  inequality~\eqref{eq2:Thm:UL_is_invariant_modified} is just a
  restatement of the local maximum principle (\ie $\sfU_i\upnp \in
  [\sfU^{\min,n},\sfU^{\max,n}]$). It is also demonstrated in the
  numerical section that the choice of  $k_i$ in
  $(\sfU_i^{\min,n},\sfU_i^{\max,n})$ should be random for the method to
  be robust when the flux $\bef$ is not strictly convex or concave.
\end{remark}

%%%%%%%%%%%%%%%%%%%%%%%%%%%%%%%%%%%%%%%%%%%%%%%%%%%%%%%%%%%%%%%%%%%%%%%%%%%%%%%%
%%%%%%%%%%%%%%%%%%%%%%%%%%%%%%%%%%%%%%%%%%%%%%%%%%%%%%%%%%%%%%%%%%%%%%%%%%%%%%%%
%%%%%%%%%%%%%%%%%%%%%%%%%%%%%%%%%%%%%%%%%%%%%%%%%%%%%%%%%%%%%%%%%%%%%%%%%%%%%%%%

\section{The $p$-system}
\label{Sec:p_system}

In this section we illustrate the greedy viscosity idea on the
one-dimensional $p$-system. The extension to the compressible Euler
equations with arbitrary equation of state will be done in the
forthcoming second part of this work.

%%%%%%%%%%%%%%%%%%%%%%%%%%%%%%%%%%%%%%%%%%%%%%%%%%%%%%%%%%%%%%%%%%%%%%%%%%%%%%%%
\subsection{The model problem}
The $p$-system is a model for isentropic gas dynamics written in Lagrangian
coordinates. The dependent variable has two components which are the
specific volume, $v$, and the velocity, $u$. The system is written as
follows:
\begin{equation}
  \label{def:p_system}
  \partial_t \begin{pmatrix}v\\ u\end{pmatrix} +
\partial_x\begin{pmatrix}-u\\ p(v)\end{pmatrix}=0,\qquad (x,t)\in \Real\CROSS\Real_+.
\end{equation}
The pressure $v\mapsto p(v)$ is assumed to be of class $C^2(\Real_+;\Real)$ and be such that
\begin{equation}
 \label{def:pressure}
p'<0, \qquad 0<p''.
\end{equation}
As an illustration, we are going to restrict the discussion to the
gamma-law, $p(v)=r v^{-\gamma}$, where $r>0$ and $\gamma > 1$.
We introduce the notation $\bu\eqq (v,u)\tr$ and define the flux $\polf(\bu) \eqq (-u,p(v))\tr$.

The admissible set for \eqref{def:p_system} is $\calA\eqq (0,\infty)\CROSS \Real$.  The p-system ($\gamma > 1$) has
two families of global Riemann invariants:
\begin{equation} \label{def:Riemann_inv}
  w_+(\bu)=u+\int_v^\infty\!\!\!\! \sqrt{-p'(\xi)}\diff \xi, \quad\mbox{and}\quad
  w_-(\bu)=u-\int_v^\infty\!\!\!\! \sqrt{-p'(\xi)}\diff \xi,
\end{equation}
and it can be shown that
\begin{equation}\label{speed_upper_bound_p_system}
\calB_{ab} := \{\bu\in \calA \st a \le w_-(\bu),
\ w_+(\bu) \le b\}
\end{equation}
is an invariant domain for the system \eqref{def:p_system} for all
$a<b\in\Real$; see \cite[Exp.~3.5, p.~597]{Hoff_1985} for a proof in
the context of parabolic regularization, or
\cite{Young_2002} for a direct proof. Note in passing that it is established in
\cite[Thm.\,2.1]{Hoff_1979} and \citep[Thm.\,4.1]{Hoff_1985} that the Lax
scheme is invariant-domain preserving for all $\calB_{ab}$.

The $p$-system has many entropy pairs. We are going to use the following
one:
\begin{equation}\label{p_system_entropy}
  \eta(\bu) = \frac12u^2 +\int_v^\infty p(\xi) \diff \xi;\qquad
  q(\bu) = u p(v).
\end{equation}
We now follow the principles explained in Algorithm~\ref{Greedy_algorithm}
to estimate a greedy viscosity.

%%%%%%%%%%%%%%%%%%%%%%%%%%%%%%%%%%%%%%%%%%%%%%%%%%%%%%%%%%%%%%%%%%%%%%%%%%%%%%%%
\subsection{Maximum wave speed}
Let us consider a left state $\bu_i\eqq (v_i,u_i)\tr$, a right state
$\bu_j\eqq (v_j,u_j)\tr$, and a one-dimensional normal direction
$n_{ij}\in\{-1,+1\}$ where $i\in\calV$ and $j\in\calI(i)$. We now describe
a procedure to compute (an upper bound of) the maximal wave speed
$\lambda_{\max}(n_{ij},\bu_i,\bu_j)$ that was introduced in
\eqref{eq:lambda_max} in \S\ref{Sec:hyperbolic_system}. One first realizes
that the Riemann problem with the flux $\polf(\bu)n_{ij}$, left data data
$(v_i,u_i)\tr$ and right data $(v_j,u_j)\tr$, is identical to the Riemann
problem with the flux $\polf(\bu)$ and data $\bu_L\eqq (v_i,n_{ij}u_i)\tr$,
$\bu_R\eqq (v_j,n_{ij}u_j)\tr$. We now use the symbol $n$ in lieu of
$n_{ij}$ and write $\lambda_{\max}(n,\bu_L,\bu_R)$ instead of
$\lambda_{\max}(n_{ij},\bu_i,\bu_j)$.

For the index $Z\in \{L,R\}$, we introduce
\begin{equation}
  f_Z(v) :=
  \begin{cases}
    -\sqrt{(p(v)-p(v_Z))(v_Z-v)},& \text{if $v\le v_Z$}
      \\
    \displaystyle\int_{v_Z}^v \sqrt{-p'(\xi)}\diff \xi, & \text{if
    $v>v_Z$}.
  \end{cases}
\end{equation}
and define $\phi(v):=f_L(v)+f_R(v)+ u_L-u_R$. The function $\phi$ is
increasing and concave with $\lim_{v\to +0}\phi(v)=-\infty$; see
\cite{Young_2002} for details. Notice that
$\lim_{v\to+\infty}\phi(v) = w_+(\bu_L)-w_-(\bu_R)$. If
$w_+(\bu_L)-w_-(\bu_R)\le 0$, then we set $v^*\eqq +\infty$ (vacuum
appears in the Riemann solution in this case). If
$w_+(\bu_L)-w_-(\bu_R)\ge 0$, the equation $\phi(v)=0$ has a unique
solution which we denote by $v^*$. Setting $v_{\min}\eqq \min(v_L,v_R)$, we
have $\phi(v_{\min}) = u_L-u_R - \sqrt{(p(v_R)-p(v_L))(v_L-v_R)}$,
and the following result is standard (see \eg \citep{Young_2002},
\citep[Lem.\,2.5]{Guermond_Popov_SINUM_2016}):
\begin{equation} \label{lambda_max_p_system}
  \lambda_{\max}(n,\bu_L,\!\bu_R)\! = \!\!
  \begin{cases}
    \sqrt{\frac{p(v_{\min})-p(v^*)}{v^*-v_{\min}}},
    & \text{if $\phi(v_{\min})> 0$,}
    \\
    \sqrt{-p'(v_{\min})},
    &\text{otherwise,}
  \end{cases}
\end{equation}
Note that $\lambda_{\max}(n,\bu_L,\bu_R)$ is a decreasing function
of $v^*$. The value of $v^*$ can be found using Newton's method
starting with a guess $v^0$ smaller than $v^*$. As $\phi$ is concave
and increasing, starting the Newton iterations on the left of $v^*$
guarantees that at each step of Newton's method the new estimate is smaller
than $v^*$, which in turn implies that the estimated maximum speed is an
upper bound for the exact maximum speed. A starting guess $v^0$ with the
above property can be computed as follows:
\begin{subequations}
  \begin{align}
    \label{eq:def_wpm_p_system}
    w_+^{\max}& \eqq \max(w_+(\bu_L),w_+(\bu_R)),\qquad
    w_-^{\min} \eqq \min(w_-(\bu_L),w_-(\bu_R)) \\
    \label{eq:def_v_0_p_system}
    v^0&\eqq
    (\gamma r)^{\frac{1}{\gamma-1}}
    \left(\frac{4}{(\gamma-1)(w_+^{\max}-w_-^{\min})}\right)^{\frac{2}{\gamma-1}}.
  \end{align}
\end{subequations}
Here, \eqref{eq:def_v_0_p_system} follows from finding the pair
$u^0\eqq (v^0,u^0)\tr$ solving $w_+(\bu^0)=w_+^{\max}$ and
$w_-(\bu^0)=w_-^{\min}$. This construction implies
\begin{equation}\label{widehat_lambda_p_system}
  \lambda_{\max}(n,\bu_L,\!\bu_R)\le
\widehat\lambda_{\max}\eqq \sqrt{\frac{p(v_{\min})-p(v^0)}{v^0-v_{\min}}}.
  \end{equation}

%%%%%%%%%%%%%%%%%%%%%%%%%%%%%%%%%%%%%%%%%%%%%%%%%%%%%%%%%%%%%%%%%%%%%%%%%%%%%%%%
\subsection{Invariant-domain property} \label{IDP_p_system}
We first compute three wave speeds to guarantee a local invariant-domain
property as in \eqref{Abstract_Psil}. Then we compute a fourth wave speed
in \S\ref{Sec:p_system_entropy_lambda} so as to ensure that a local entropy
inequality holds for the above-defined entropy pair;
see~\eqref{Abstract_etae}. Recall that
\begin{align}
  \overline{\bu}_{LR}(\lambda) =
  \frac12  \begin{pmatrix}
    v_L+v_R + \frac{1}{\lambda}(u_R-u_L) \\
    u_L + u_R  -  \frac{1}{\lambda}(p(v_R)-p(v_L))
  \end{pmatrix}.
\end{align}
We introduce
\begin{gather}
  \Psi_1(\bu) \eqq v,\qquad   \Psi_2(\bu) \eqq w_+^{\max} - w_+(\bu),\qquad
  \Psi_3(\bu) \eqq  w_-(\bu) - w_-^{\min},
\end{gather}
where $w_+^{\max}$ and $w_-^{\min}$ are defined in
\eqref{eq:def_wpm_p_system}.  Observe that $\Psi_1$ is concave and
$\Psi_2$ and $\Psi_3$ are both strictly concave due
to~\eqref{def:pressure}. We define $\calB_0\eqq \Real^2$,
$\calB_1\eqq \{\bu\in\Real^2\st \Psi_1(v)>0\} = \calA$,
$\calB_2 \eqq \{\bu \in\calB_1\st \Psi_2(\bu) \ge 0\}$, and
$\calB_3 \eqq \{\bu \in\calB_2\st \Psi_3(\bu) \ge 0\}$. It is
necessary to introduce $\Psi_1$ and $\calB_1=\calA$ to make sure that the
domain of definition of $\Psi_2$ and $\Psi_3$ is $\calA$.

If $\bu_L=\bu_R$, then  $\overline\bu_{LR}(\lambda) =
\bu_L=\bu_R$ for all $\lambda>0$. In this case, we take
$\lambda_1=\lambda_2=\lambda_3=\lambda_\epsilon$.
Let us now assume that  $\bu_L\not=\bu_R$.
The smallest wave speed $\lambda_1$, greater than or equal to $\lambda_\epsilon$,
that ensures $\Psi_1(\overline\bu_{LR}(\lambda))> 0$ for all
$\lambda> \lambda_1$ is given by
\begin{equation}
  \lambda_1 =\max\left(\frac{u_L-u_R}{v_L+v_R},\lambda_\epsilon\right).
\end{equation}
Now we estimate $\lambda_2$. If
$\Psi_2(\overline\bu_{LR}(\lambda_1))\ge 0$, then we set
$\lambda_2 = \lambda_1$. If $\Psi_2(\overline\bu_{LR}(\lambda_1))< 0$
there are two cases. If $u_R-u_L\ge 0$ and $-(p(v_R)-p(v_L))\le 0$, we
have
$\Psi_2(\overline\bu_{LR}(\lambda))\ge
\Psi_2(\frac12(\bu_{L}+\bu_R))\ge 0$ for all $\lambda>0$ and we set
$\lambda_2\eqq \lambda_1$. Otherwise, we observe that the curve
$w_+(\bu) = w_+^{\max}$ has a horizontal asymptote given by
$\{u= w_+^{\max}\}$ and a vertical asymptote given by $\{v=0\}$ and
the condition ($u_R-u_L< 0$ or $-(p(v_R)-p(v_L))> 0$) 
implies that the equation $\Psi_2(\overline\bu_{LR}(\lambda))= 0$ has
a unique positive solution, $\lambda_2^*$, which can be computed using
an iterative method, and we set $\lambda_2=\lambda_2^*$; we omit the
details for brevity. The argument to estimate $\lambda_3$ is
analogous: If $\Psi_3(\overline\bu_{LR}(\lambda_2)) \ge 0$, then we
set $\lambda_3 = \lambda_2$. Otherwise, we observe that the curve
$w_-(\bu) = w_-^{\min}$ has a horizontal asymptote given by
$\{u= w_-^{\min}\}$ and a vertical asymptote given by $\{v=0\}$. Hence
if $u_R-u_L\ge 0$ and $-(p(v_R)-p(v_L))\ge 0$, we have
$\Psi_3(\overline\bu_{LR}(\lambda)) \ge
\Psi_3(\frac12(\bu_{L}+\bu_R))> 0$ for all $\lambda>0$ and we set
$\lambda_3\eqq \lambda_2$. Otherwise the equation
$\Psi_3(\overline\bu_{LR}(\lambda))= 0$ has a unique positive
solution, $\lambda_3^*$, which can be computed using an iterative
method, and we set $\lambda_3=\lambda_3^*$. As asserted in
Lemma~\ref{Lem:def_lambdas}, the process described above guarantees
that
$\overline\bu_{LR}(\lambda) \in \calB_3\eqq \{\bu \in \calA\st
\Psi_2(\bu)\ge 0, \ \Psi_3(\bu)\ge 0 \}$ for all
$\lambda\ge \lambda_3$.

\subsection{Wave speed based on the entropy inequality}
\label{Sec:p_system_entropy_lambda}

We now estimate a wave speed associated with one entropy inequality.
The entropy functional in this case is
\begin{equation}
  \Phi_e(t) :=\eta\left(\overline\bu_{LR}(\tfrac{1}{t})\right)-
  \frac{1}{2}(\eta(\bu_L)+\eta(\bu_R)) + \frac{t}{2}(q(\bu_R) -
  q(\bu_L)),
  \label{def_Phi_p_system}
\end{equation}
where $\eta$ and $\bq$ are defined in \eqref{p_system_entropy}. We have
$\eta(\bu) = \frac12 u^2 -  \frac{1}{1-\gamma} v p(v)$ for the pressure
gamma-law.

If $\bu_L=\bu_R$, then  $\overline\bu_{LR}(\lambda) =
\bu_L=\bu_R$ for all $\lambda>0$ and $\Phi_e(t)=0$ for all $t\ge 0$. In
this case, we take $\lambda_e=\lambda_3$.
If $\bu_L\not=\bu_R$, we compute $\lambda_e$
as defined in \eqref{Abstract_etae}. More precisely,
if $\Phi_e(\frac{1}{\lambda_3})\le 0$, then we set
$\lambda_e=\lambda_3$. Otherwise,
we observe that the equation $\Phi_e(\frac{1}{\lambda})=0$ has a
unique solution in $[\frac{1}{\lambda_{LR}^\sharp},\frac{1}{\lambda_3})$  because $\eta$
defined in \eqref{p_system_entropy} is strictly convex and we also have
established in \eqref{Ass_Ulbd_max:2} that $\Phi_e(\frac{1}{\lambda_{LR}^\sharp})\le 0$.
Finally, we set $\lambda_e = \max(\lambda_e,\lambda_3).$
The greedy wave speed is obtained by setting
$\lambda\upgrdy(n,\bu_L,\bu_R)\eqq \lambda_e$. This
algorithm is illustrated numerically in
\S\ref{Sec:numerical_illustration_p_system}.

%%%%%%%%%%%%%%%%%%%%%%%%%%%%%%%%%%%%%%%%%%%%%%%%%%%%%%%%%%%%%%%%%%%%%%%%%%%%%%%%
%%%%%%%%%%%%%%%%%%%%%%%%%%%%%%%%%%%%%%%%%%%%%%%%%%%%%%%%%%%%%%%%%%%%%%%%%%%%%%%%
%%%%%%%%%%%%%%%%%%%%%%%%%%%%%%%%%%%%%%%%%%%%%%%%%%%%%%%%%%%%%%%%%%%%%%%%%%%%%%%%

\section{Numerical illustrations with scalar conservation equations}
\label{Sec:Numerical_illustrations}

We start by illustrating the method for scalar conservation equations. To
test the robustness of the method, we choose problems with fluxes that are
not strictly convex and contain \emph{sonic points}. Methods that
underestimate the maximum wave speed (or just enforce the maximum
principle) tend to fail when applied to this type of problems.

Here, we numerically show that computing the viscosity so as to enforce
local entropy inequalities is sufficient to select the entropy solution
provided that the family of entropies is rich enough. All the computations
are done with continuous $\polP_1$ finite elements and we take
$\epsilon=10^{-8}$ in \eqref{def_lambda_0}. The time stepping is done with
the three stages, third-order, strong stability preserving Runge Kutta
method \citep{Shu_Osher1988}. The time step is computed by using the
expression $\dt_n = \frac{\text{\cfl}}{2}\max_{i\in\calV}
m_i/\sum_{j\in\calI(i)^*} d_{ij}\upgrdyn$.

%%%%%%%%%%%%%%%%%%%%%%%%%%%%%%%%%%%%%%%%%%%%%%%%%%%%%%%%%%%%%%%%%%%%%%%%%%%%%%%%
\subsection{Piecewise linear flux}
\label{Sec:Piecewise_linear_flux}
We consider a Riemann problem in one space dimension for the scalar
conservation equation $\partial_t u + \partial_x f(u) =0$ using the scalar
flux $f(v)=2-v$ if $v\le 2$ and $f(v)=2v-4$ otherwise. The initial data is
$u_0(x)=1$ if $x\le 0$ and $u_0(x)=3$ otherwise. This flux is convex and
Lipschitz, but it is not strictly convex: the velocity is piecewise
constant and discontinuous. This class of problems is thoroughly
investigated in \cite{Petrova_Popov_1999}. The solution is
\begin{equation}
  u(x,t) =
  \begin{cases}
    1 & \text{if $x \le -t$}\\
    2 & \text{if $-t < x \le 2t$}\\
    3 & \text{if $2t < x$}.
  \end{cases}
\end{equation}
The solution is composed of two contact waves (\ie the characteristics do
not cross) separated by an expansion wave. One contact wave moves to the
left at speed $-1$, the other moves to the right at speed 2.

\begin{figure}[t]
\centering
\subfloat[]{\includegraphics[width=0.35\textwidth]{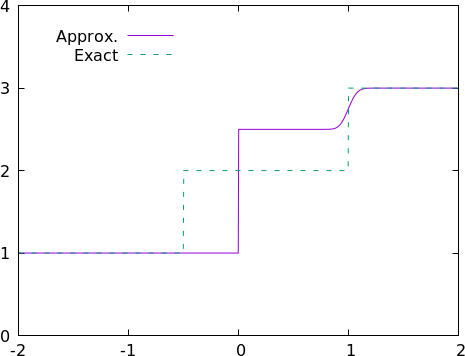}}
\hfil
\subfloat[]{\includegraphics[width=0.35\textwidth]{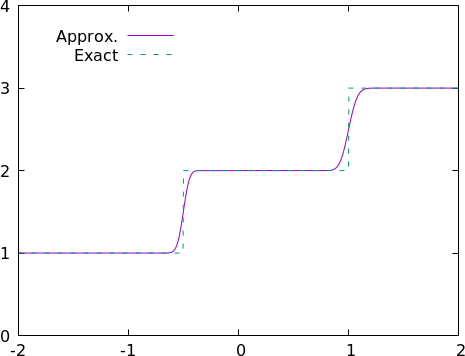}}
\caption{Approximation of a scalar conservation equation with piecewise
  linear flux: (a) viscosity solely based on $\lambda_{12}$ given in
  \eqref{def_lambda_12}; (b) viscosity based on an entropy inequality,
 Eq.~\eqref{Lambda_Scalar_Entrop_Ineq} with the choice $k_i=\frac12
  (\sfU_i^{\min,n}+\sfU_i^{\max,n}).$}
  \label{Fig:1D_roe_average}
\end{figure}

This example is meant to demonstrate that only using the wave speed
$\lambda_{12}$ defined in \eqref{def_lambda_12} to construct the graph
viscosity (\ie only using the Roe average) is not robust even in a case as
simple as the one above. Using the wave speed $\lambda_{12}$ guarantees
that the maximum principle locally holds, but the approximation may
converge to a nonentropic weak solution. We illustrate this phenomenon by
applying the algorithm described in the paper over the domain
$\Dom=(-2,2)$ using uniform meshes. The solution is computed to a final
time $t=0.5$ using \cfl=0.75. We show in the left panel of
Figure~\ref{Fig:1D_roe_average} the solution obtained with the viscosity
computed by using only $\lambda_{12}$. The graph of the exact solution is
shown with a dashed line. We observe that the approximate solution does not
converge to the exact solution. The leftmost discontinuity in the
approximate solution is stationary instead of moving to the left at speed
$-1$. The right panel shows the approximate solution using
definition~\eqref{Lambda_Scalar_Entrop_Ineq} for the wave speed with $k_i =
\frac12 (\sfU_i^{\min,n}+\sfU_i^{\min,n})$ for every $i\in\calV$. We have
verified that the method using this
definition for the wave speed converges with the expected rate (tables not shown
here for brevity).

%%%%%%%%%%%%%%%%%%%%%%%%%%%%%%%%%%%%%%%%%%%%%%%%%%%%%%%%%%%%%%%%%%%%%%%%%%%%%%%%
\subsection{1D non-convex flux}
We now consider a Riemann problem in one space dimension  using the scalar
flux $f(v)=\sin(v)$. The initial data is $u_0(x)=(2+a)\pi$ if $x<0$ and
$u_0(x) = b\pi$ otherwise. Here, $a\in [\frac12,1]$ and $b\in [0,\frac12]$
are two chosen parameters. Note that the flux is neither convex nor concave
over the interval $[b\pi ,(2+a)\pi]$. Since $(2+a)\pi> b\pi$, the solution
is obtained by replacing the flux by its upper concave envelope which is
$\uconcave{f}(v)= \sin(v)$ for $v\in [b\pi, \frac12 \pi]$,
$\uconcave{f}(v)=1$ for $v\in [\frac12 \pi,\frac52 \pi]$, and
$\uconcave{f}(v)= \sin(v)$ for $v\in [\frac52 \pi,(2+a)\pi]$ (see, \eg
\cite[Lem.\,3.1]{Dafermos_1972} and \cite[\S2.2]{Holden_Risebro_2015}).
We note that the entire the interval $v\in [\frac12 \pi,\frac52 \pi]$ is
composed of sonic points. The exact solution is given by
\begin{equation}
  u(x,t) =
  \begin{cases}
    (2+a)\pi & \text{if $x\le t \cos((2+a)\pi)$}\\
    3\pi - \arccos(|x/t|) & \text{if $ t \cos((2+a)\pi) < x \le 0$}\\
    \arccos(x/t) &  \text{if $ 0 < x \le t \cos(b\pi)$}\\
    b \pi  &  \text{if $t \cos(b\pi)< x$}.
  \end{cases}%
  \label{ent_sol_1d_kpp}
\end{equation}
It is a composite wave composed of an expansion followed by a stationary
shock followed by a second expansion. The numerical tests reported
below are done with  $b=0$ and
$a=1$ over the domain $\Dom=(-1,1)$.

Here again, tests done with the graph viscosity solely based on the Roe
average $\lambda_{12}$ yields a method that is not robust (figures and
tables are not reported for brevity).  We observe that the approximate
solution is a stationary shock for every mesh refinement (\ie the
initial data does not evolve), which is clearly not the entropy
solution.  One can artificially try to avoid this problem by
initializing the approximate solution at $t_0>0$ using the exact
solution~\eqref{ent_sol_1d_kpp}.  If the mesh does not have a vertex
located at $\{0\}$, then convergence starts only when the mesh size is
less that $t_0$.  On the other hand, we observe convergence with no
pre-asymptotic range for every positive value of $t_0$ when the mesh
has a vertex located at $\{0\}$.  This behavior illustrates well the
lack of robustness of methods that are solely based on the wave speed
$\lambda_{12}$.

We now test the method based on the wave speed computed by using
\eqref{Lambda_Scalar_Entrop_Ineq}. The tests are done with
$\text{CFL}=0.5$. The relative errors in the $L^1$-norm and
$L^2$-norm are computed at $t=0.8$.  We test two strategies to select
the Kr\v{u}zkov entropy for each degree of freedom $i\in\calV$. The
first strategy consists of setting
$k_i= \theta \sfU_i^{\min,n}+(1-\theta)\sfU_i^{\max,n}$ where
$\theta=\frac12$. The second strategy consists of setting
$k_i= \theta_i \sfU_i^{\min,n}+(1-\theta_i)\sfU_i^{\max,n}$, where
$\theta_i\in(0,1)$ is a uniformly distributed random number changing
at every grid point $i\in\calV$.

\begin{table}[ht] \small \centering
  \begin{tabular}{rcccccccc}
    \toprule
    & \multicolumn{4}{c}{Random entropy} &  \multicolumn{4}{c}{Average entropy}  \\
    \cmidrule(lr){2-5} \cmidrule(lr){6-9}
    \# dofs & $\delta^1(t)$& rate & $\delta^2(t)$ & rate&$\delta^{1}(t)$
    &rate&$\delta^{2}(t)$ &rate \\[0.3em]
      51 & 1.96E-02 & --  & 2.21E-02  & --   & 2.41E-02 & --   & 2.30E-02 & --   \\
     101 & 1.39E-02 & 0.49 & 1.65E-02 & 0.42 & 1.81E-02 & 0.41 & 1.77E-02 & 0.38 \\
     201 & 9.17E-03 & 0.60 & 1.13E-02 & 0.55 & 1.33E-02 & 0.45 & 1.38E-02 & 0.35 \\
     401 & 5.87E-03 & 0.64 & 8.25E-03 & 0.45 & 9.81E-03 & 0.43 & 1.14E-02 & 0.28 \\
     801 & 3.66E-03 & 0.68 & 5.76E-03 & 0.52 & 7.54E-03 & 0.38 & 9.89E-03 & 0.20 \\
    1601 & 2.24E-03 & 0.71 & 4.15E-03 & 0.47 & 6.10E-03 & 0.31 & 9.06E-03 & 0.13 \\
    3201 & 1.38E-03 & 0.70 & 2.89E-03 & 0.52 & 5.20E-03 & 0.23 & 8.62E-03 & 0.07 \\
    6401 & 8.50E-04 & 0.70 & 2.06E-03 & 0.49 & 4.65E-03 & 0.16 & 8.39E-03 & 0.04 \\
    \bottomrule
  \end{tabular}
  \caption{1D two-sonic point problem. The second and fourth columns show
    relative errors in the $L^1$-norm and the $L^2$-norms using a random
    Kr\v{u}zkov entropy with $k= \theta u_L+(1-\theta) u_R$, where
    $\theta\in(0,1)$ is a uniformly distributed random value.
    The sixth and eight columns report relative errors in the $L^1$-norm
    and the $L^2$-norms obtained for the average Kr\v{u}zkov entropy with $k=
    \frac12(u_L+ u_R)$.}
  \label{Tab:1D_KPP}
\end{table}

When using the first strategy with fixed $\theta=\frac12$ we observe
exactly the same problems as reported above when only using
$\lambda_{12}$.  Irrespective of the location of the grid points, the
approximate solution is a stationary shock when one initializes the
approximate solution with the exact solution at $t_0=0$. Initializing
with the exact solution~\eqref{ent_sol_1d_kpp} at
$t_0=10^{-8}$ still produces a stationary shock when the point
$\{x=0\}$ is not a vertex of the mesh, but a non trivial solution is
obtained when the point $\{x=0\}$ is a vertex of the mesh. We show in
the right part of Table \ref{Tab:1D_KPP} convergence results using
$t_0=10^{-8}$ and uniform meshes with odd numbers of grid points.  We
observe some kind of convergence on coarse meshes, but eventually the error
stalls and stagnates as the mesh is further refined. We have observed this
behavior for every constant value of $\theta$. This is highly counter
intuitive because the viscosity based on~\eqref{Lambda_Scalar_Entrop_Ineq}
is strictly larger than $\lambda_{12}$, and we have observed in the above
paragraph that the approximate solution using $\lambda_{12}$ converges to
the entropy solution when the point $\{x=0\}$ is a mesh vertex. Here again,
we observe a clear lack of robustness even when the wave speed is augmented
so as to guarantee one ``entropy fix'' per grid point.

We now discuss what happens when the Kr\v{u}zkov entropy is randomly
chosen.  All the problem mentioned above disappear when
$\theta_i\in(0,1)$ is randomly chosen at every grid point. The method
convergences whether there is a grid point at $\{0\}$ or not and whatever
the initial time. In particular there is no problem setting $t_0=0$.  We
show convergence tests in the left panel of Table~\ref{Tab:1D_KPP} with
$t_0=0$. To be able to compare with the results displayed in the right part
of the table, we have use the same meshes. The method is now clearly
convergent and converges with the expected rates.

The conclusion of this section is that the method based on the greedy wave
speed computed by using \eqref{Lambda_Scalar_Entrop_Ineq} with random
Kr\v{u}zkov entropies is robust.

\begin{figure}[ht]
  \centering
  \subfloat[Solution at $t=0.8$]{%
    \begin{minipage}[c]{0.44\textwidth}
      \includegraphics[width=\textwidth]{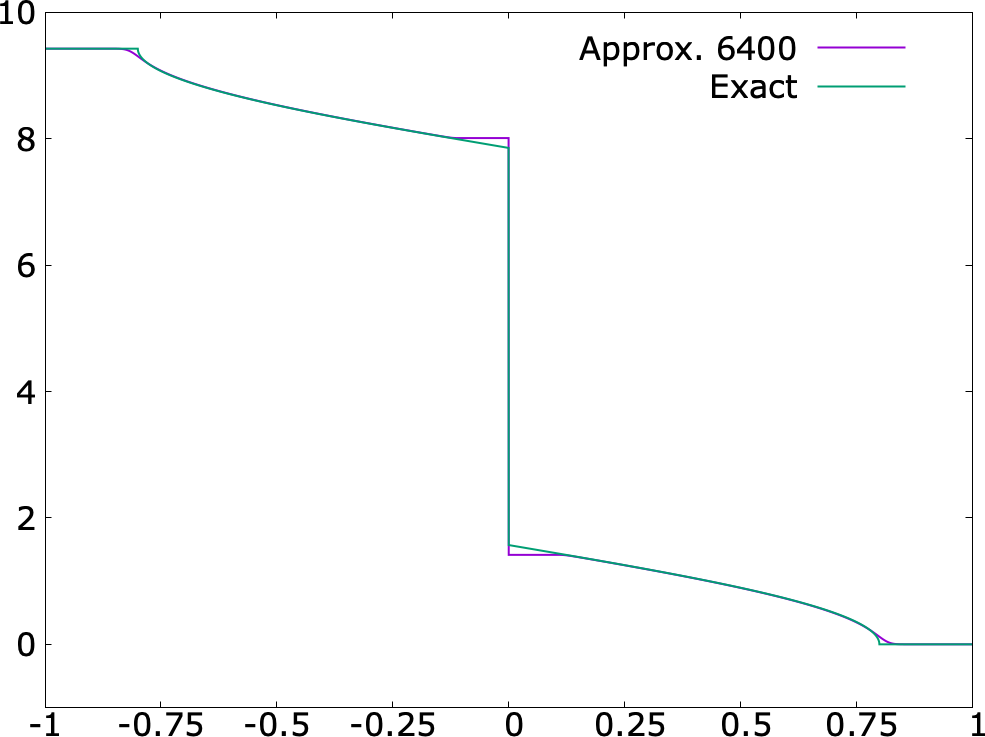}
    \end{minipage}}
  \hspace{0.25em}
  \subfloat[Convergence table]{%
    \small
    \begin{tabular}{rcccc}
      \toprule
      & \multicolumn{4}{c}{Square entropy $\eta(v)=\frac12v^2$} \\
      \cmidrule(lr){2-5}
      \# dofs & $\delta^1(t)$& rate & $\delta^2(t)$ & rate \\[0.25em]
          50 & 2.22E-02 & --  & 2.18E-02 & --   \\
         100 & 1.63E-02 & 0.45 & 1.64E-02 & 0.41 \\
         200 & 1.15E-02 & 0.50 & 1.23E-02 & 0.42 \\
         400 & 8.08E-03 & 0.51 & 9.44E-03 & 0.38 \\
         800 & 5.82E-03 & 0.47 & 7.61E-03 & 0.31 \\
        1600 & 4.38E-03 & 0.41 & 6.50E-03 & 0.23 \\
        3200 & 3.48E-03 & 0.33 & 5.87E-03 & 0.15 \\
        6400 & 2.92E-03 & 0.25 & 5.52E-03 & 0.09 \\
     \bottomrule
   \end{tabular}}
   \caption{1D two-sonic point problem computed with the square
     entropy $\eta(v)=\frac12v^2$. The ``entropy stable'' method does not converge to the
     entropy solution.}
  \label{fig:square_entropy}
\end{figure}

\begin{remark}[Robustness and ``entropy stability'']
  The numerical tests performed in this section demonstrates that
  robustness comes from randomness of the Kr\v{u}zkov entropy. Note in
  passing that this series of tests casts doubt on the robustness of
  methods that are called \emph{entropy stable} in the literature. Since
  these methods enforce only \emph{one fixed} global entropy inequality (at
  the semi-discrete level), one may wonder whether they produce
  approximations that converge to the right solution for the above
  one-dimensional problem. In order to provide some numerical evidence in
  this matter, we adjust our method as introduced in
  \S\ref{Sec:entropy_ineq_lambda_scalar} for the entropy
  $\eta(v)=\frac12v^2$ which is usually invoked in the literature dedicated
  to entropy stable methods. Redoing the computations in the proof of
  Lemma~\ref{Lem:Scalar_entropy_speed} with the square entropy gives
  $\lambda= (2 a b+d+\sqrt{\Delta})/(2(c-a^2))$ with $a\eqq
  \frac12(u_L+u_R)$, $b \eqq \frac12(\bef(u_L)-\bef(u_R))\SCAL\bn$, $c \eqq
  \eta(u_L) + \eta(u_R)$, $d = (\bq(u_R) - \bq(u_L))\SCAL \bn$, $\Delta
  \eqq (2 a b+d)^2 -4 b^2(a^2-c)$. The method thus produced is locally and
  globally entropy stable with respect to $\eta(v)=\frac12v^2$, \ie
  \eqref{eq2:Thm:UL_is_invariant_greedy} holds. Convergence tests with this
  method are reported in Figure~\ref{fig:square_entropy}. These tests show
  that the approximation does not converge to the entropy
  solution~\eqref{ent_sol_1d_kpp}. The convergence behavior is strange as
  the approximation seems to converge over a large pre-asymptotic range,
  but eventually, when the mesh is very fine, the approximation converges
  to a weak solution that is not the entropy solution. In conclusion, the
  method is definitely entropy stable for the square entropy but it is not
  convergent for non-convex fluxes; hence, it is not robust.
\end{remark}

%%%%%%%%%%%%%%%%%%%%%%%%%%%%%%%%%%%%%%%%%%%%%%%%%%%%%%%%%%%%%%%%%%%%%%%%%%%%%%%%
\subsection{The 2D KPP problem}
\begin{figure}[h]
  \centering
  \subfloat[$\lambda_{12}$ with \eqref{def_lambda_12}]%
    {\includegraphics[width=0.31\textwidth]{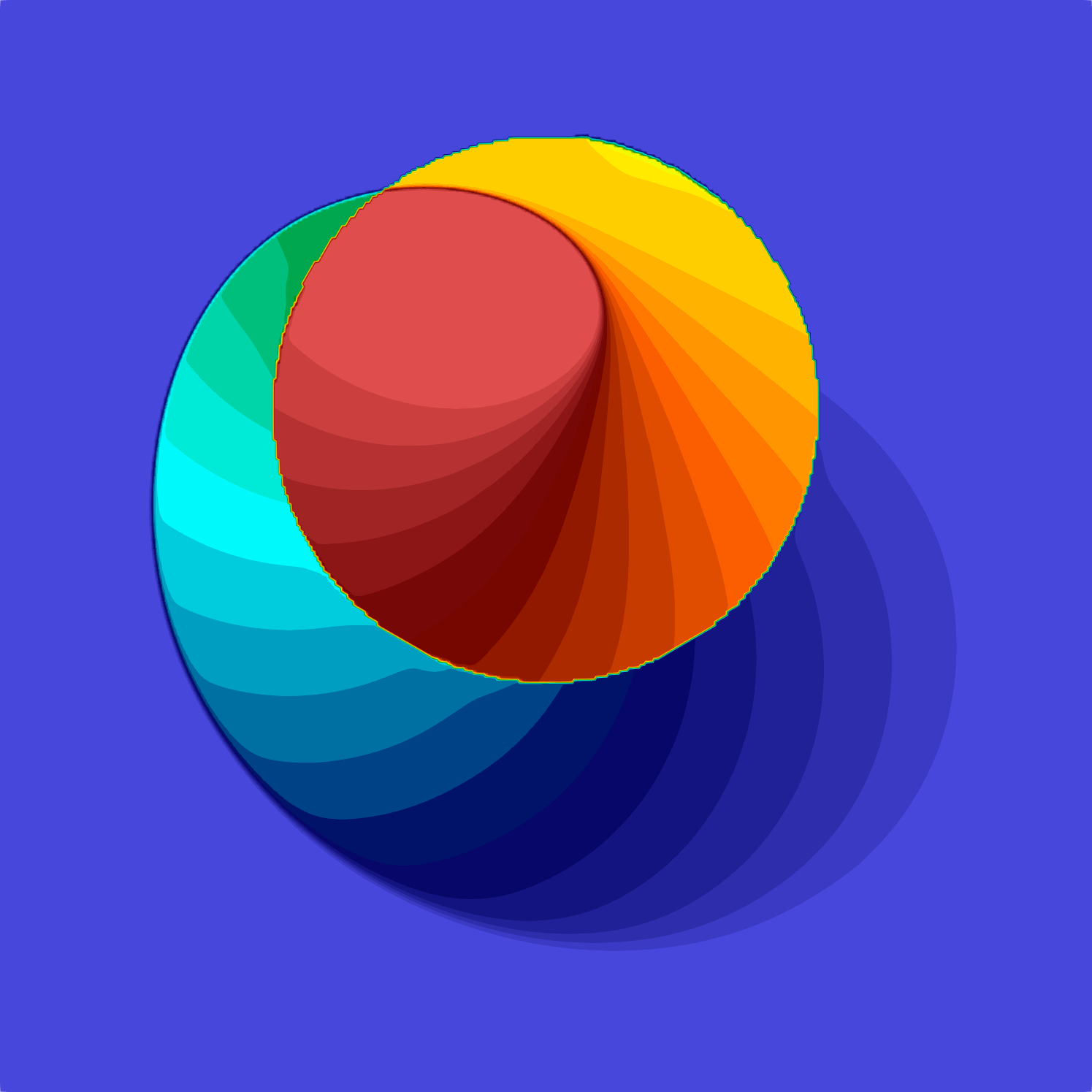}}
  \hspace{0.3em}
  \subfloat[$\lambda^{\text{grdy}}$ with \eqref{Lambda_Scalar_Entrop_Ineq}, $\theta_i=\tfrac12$]%
    {\includegraphics[width=0.31\textwidth]{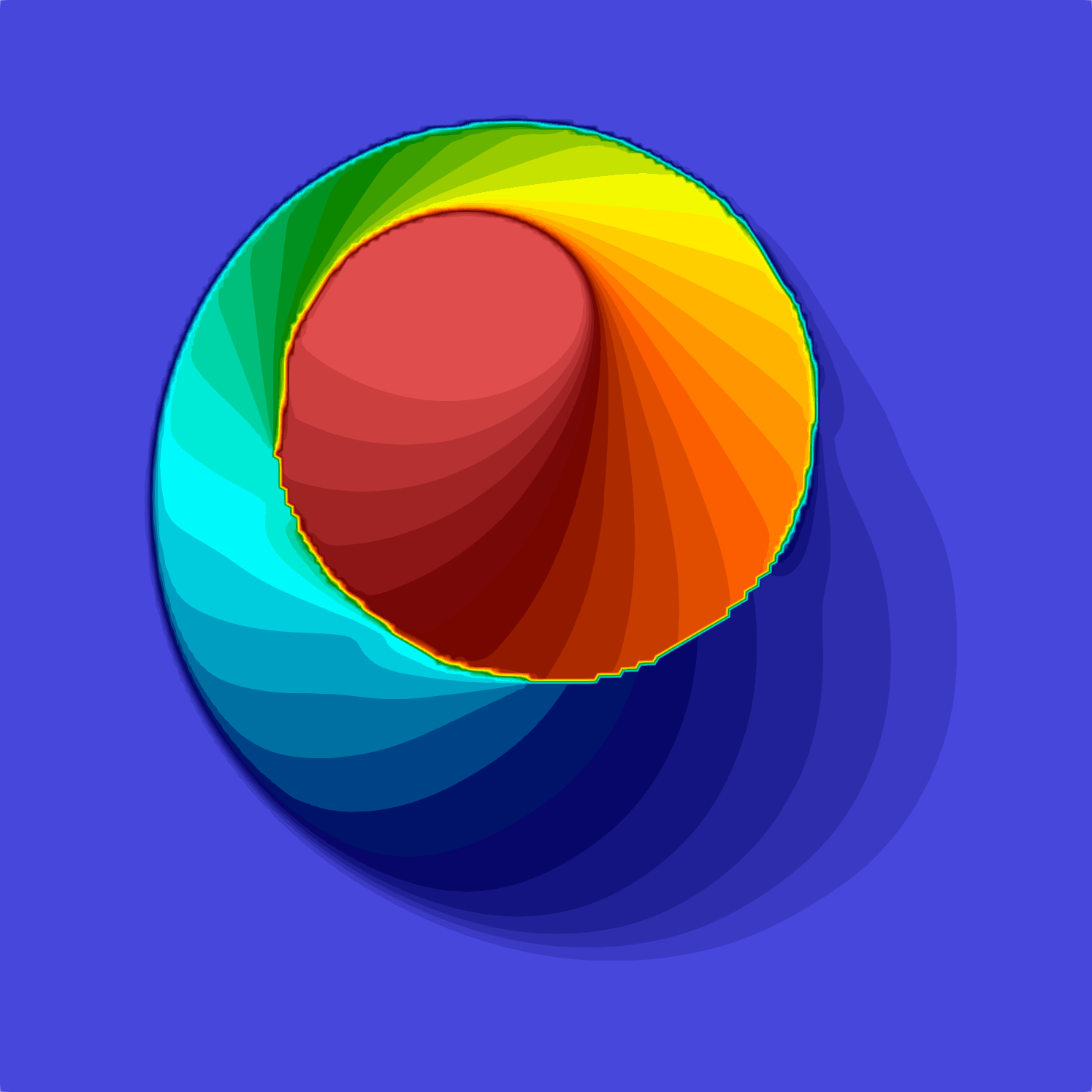}}
  \hspace{0.3em}
  \subfloat[$\lambda^{\text{grdy}}$ with \eqref{Lambda_Scalar_Entrop_Ineq}, $\theta_i$ random]%
    {\includegraphics[width=0.31\textwidth]{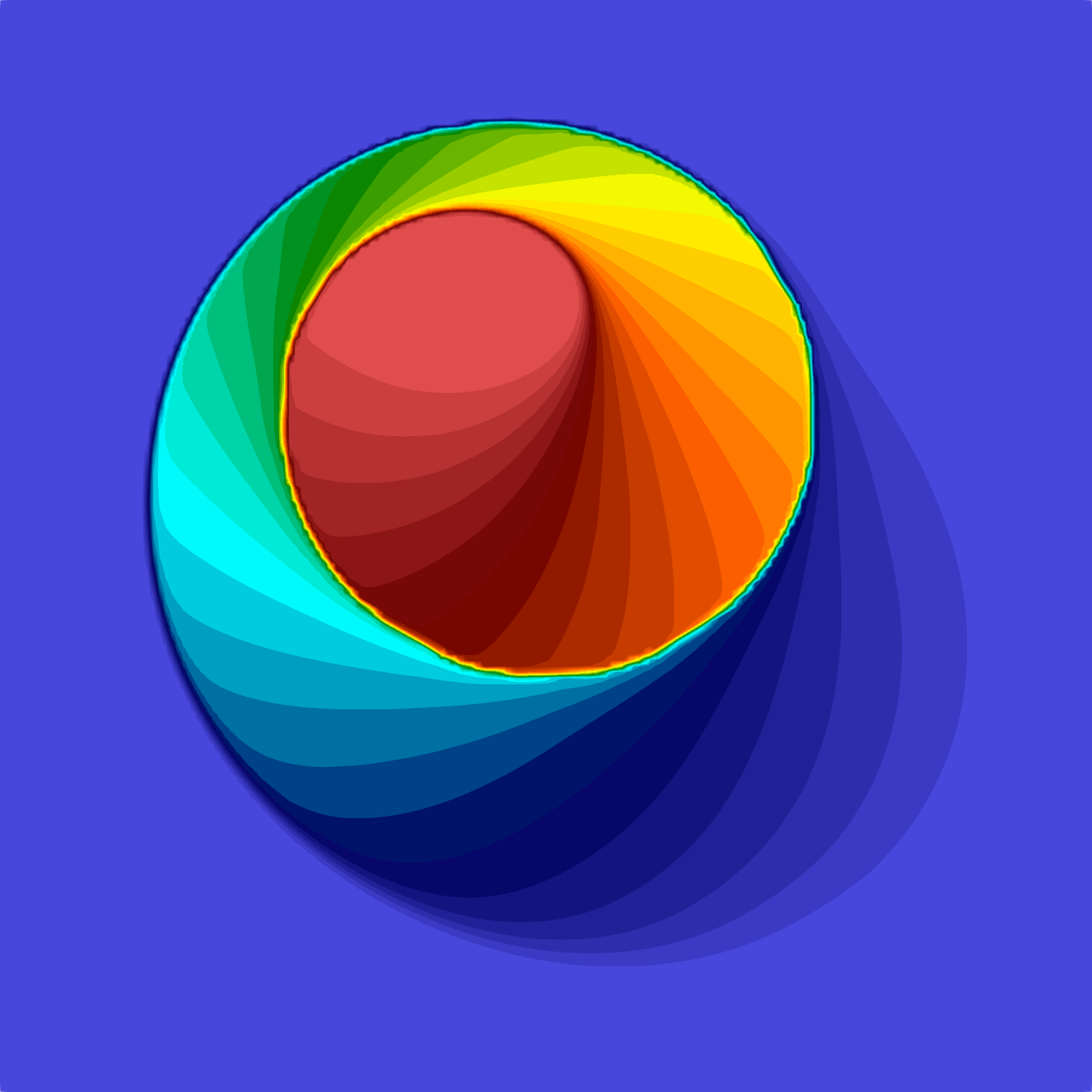}}
  \caption{2D KPP problem with $\polP_1$ elements on nonuniform Delaunay
    mesh (118850 grid points) at $t=1$, $\text{CFL}=0.5$, computed with
    three different strategies: (a) $\lambda_{\max}= \lambda_{12}$ using
    \eqref{def_lambda_12}; (b) wave speed $\lambda^{\text{grdy}}$ computed
    with~\eqref{Lambda_Scalar_Entrop_Ineq} using $k_i= \theta
    \sfU_i^{\min,n}+(1-\theta)\sfU_i^{\max,n}$ with $\theta=\frac12$; (c)
    wave speed $\lambda^{\text{grdy}}$ computed
    with~\eqref{Lambda_Scalar_Entrop_Ineq} using $k_i= \theta_i
    \sfU_i^{\min,n}+(1-\theta_i)\sfU_i^{\max,n}$ where $\theta_i\in(0,1)$
    is a uniformly random number changing for every $i\in\calV$. Only the
    solution in the right panel is the correct entropy solution.}
  \label{Fig:2D_KPP}
\end{figure}
We finish our numerical examples by solving a two-dimensional scalar
conservation equation with the non-convex flux $\bef(u):= (\sin u, \cos
u)\tr$
\begin{equation}
  \partial_t u + \DIV\bef(u) = 0, \quad
  u(\bx,0) = u_0(\bx)=
    \begin{cases}
      \tfrac{14\pi}{4} &\mbox{if } \sqrt{x^2+y^2} \leq 1 \\
      \tfrac{\pi}{4} &\mbox{otherwise},
    \end{cases}
  \label{test_KPP}
\end{equation}
in the computational domain $\Dom=[-2,2]\CROSS[-2.5,1.5]$. The problem
was originally proposed in \cite{KPP_2007}. The solution has a
two-dimensional composite wave structure which high-order numerical
schemes have difficulties to capture correctly. We approximate the solution
with continuous $\polP_1$ finite elements on nonuniform Delaunay
triangulations up to a final time of $t=1$. We show in
Figure~\ref{Fig:2D_KPP} three results computed on a mesh with 118850
grid points with {\cfl}=0.5. The solution shown in the leftmost panel
is obtained by only using the wave speed $\lambda_{12}$ for computing the
greedy viscosity. The solution in the midle panel is obtained with the the
wave speed \eqref{Lambda_Scalar_Entrop_Ineq} and the Kr\v{u}zkov entropy
using $k_i= \theta \sfU_i^{\min,n}+(1-\theta)\sfU_i^{\max,n}$ with
$\theta=\frac12$. The solution in the rightmost panel is obtained with the
the wave speed \eqref{Lambda_Scalar_Entrop_Ineq} and the Kr\v{u}zkov
entropy using $k_i= \theta_i \sfU_i^{\min,n}+(1-\theta_i)\sfU_i^{\max,n}$
where $\theta_i$ is a random number changing for every $i\in\calV$. One may
be mislead thinking that the solution in the middle panel is correct, but
the only approximation that converges correctly is the one using the random
entropy.

So, here again, our conclusion for scalar conservation equations is that
robustness can be achieved for methods based on the greedy wave speed
\eqref{Lambda_Scalar_Entrop_Ineq} provided the Kr\v{u}zkov entropies are
chosen randomly. Any other choice is not robust.

%%%%%%%%%%%%%%%%%%%%%%%%%%%%%%%%%%%%%%%%%%%%%%%%%%%%%%%%%%%%%%%%%%%%%%%%%%%%%%%%
%%%%%%%%%%%%%%%%%%%%%%%%%%%%%%%%%%%%%%%%%%%%%%%%%%%%%%%%%%%%%%%%%%%%%%%%%%%%%%%%
%%%%%%%%%%%%%%%%%%%%%%%%%%%%%%%%%%%%%%%%%%%%%%%%%%%%%%%%%%%%%%%%%%%%%%%%%%%%%%%%

\section{p-System} \label{Sec:numerical_illustration_p_system}
%
%Here we test the method when the left and right states sit on a
%rarefaction or on a shock curve.

We test the method on the p-system using the equation of state $p(v) =
\frac{1}{\gamma} v^\gamma$ with $\gamma=3$.
We consider a Riemann problem with left state
$\bu_L=(v_L,\sqrt{(1-v_L)(p(v_L)-p(1))}$ and right state
$\bu_R=(v_R,-\sqrt{(1-v_R)(p(v_R)-p(1))}$. The solution is composed of
two shock waves when $v_L,v_R> 1$, and in this case $v^*=1$, $u^*=0$.
For this test we set $v_L = 1.5$ and $v_R=1000$. The left shock is
weak and fast moving; the shock speed is close to $-0.6849$. The
right shock is strong and slow; the shock speed is close to $1.827\CROSS 10^{-2}$.

The simulations are done in the computational domain $\Dom\eqq
(0,1)$. The initial data is $\bu_0(x)=\bu_L$ if $x<0.8$ and
$\bu_0(x)=\bu_R$ otherwise. The relative error in the $L^1$-norm is
computed at $t=0.7$. The relative error is the sum of the relative
error on $v$ plus the relative error on $u$. Convergences test are
done on a sequence of uniform meshes starting from $51$ grid points to
$1601$ grid points. The results are shown in
Table~\ref{Tab:P_system}. The results in the first column are obtained
by using the upper wave speed estimate $\widehat\lambda_{\max}$ given
in \eqref{widehat_lambda_p_system}.
Those shown in the second column are obtained by using
$\lambda_{\max}$ as defined in \eqref{lambda_max_p_system} where $v^*$ is computed with a
Newton method with $10^{-10}$ tolerance. Those shown in the right
column are obtain with the greedy viscosity $\lambda^{\text{grdy}}$
defined in \S\ref{IDP_p_system}-\ref{Sec:p_system_entropy_lambda}.

\begin{table} \small\centering
  \begin{tabular}{rcccccc}
    \toprule
    & \multicolumn{2}{c}{$\widehat\lambda_{\max}$}
    & \multicolumn{2}{c}{$\lambda_{\max}$}
    & \multicolumn{2}{c}{$\lambda^{\text{grdy}}$} \\
    \cmidrule(lr){2-3} \cmidrule(lr){4-5} \cmidrule(lr){6-7}
    \# dofs & $\delta^1(t)$ & rate & $\delta^1(t)$ & rate & $\delta^1(t)$ & rate  \\[0.25em]
      51 & 3.33E-01  &   -- & 1.93E-01  &   -- & 1.31E-01  &   --  \\
     101 & 2.41E-01  & 0.47 & 1.57E-01  & 0.30 & 1.18E-01  & 0.15  \\
     201 & 1.41E-01  & 0.78 & 6.58E-02  & 1.25 & 4.93E-02  & 1.25  \\
     401 & 7.59E-02  & 0.89 & 4.42E-02  & 0.58 & 3.65E-02  & 0.43  \\
     801 & 3.64E-02  & 1.06 & 2.09E-02  & 1.08 & 1.77E-02  & 1.05  \\
    1601 & 1.70E-02  & 1.10 & 9.07E-03  & 1.20 & 7.76E-03  & 1.19  \\
    \bottomrule
  \end{tabular}
  \caption{Convergence tests for the $p$ system for various choices of
  wave speed estimate.}\label{Tab:P_system}
\end{table}

We show in Figure~\ref{Fig:p_system} the graph of the  $u$ component at
the final time $t=0.7$. In the left panel the approximation is done
with 101 uniform grid points. We show a closer view of the plateau
separating the two shocks in the right panel. The number of grid
points used in each case is: 101 in the top right panel; 401 in middle right panel;
and 1600 in the bottom right panel.
This series of simulations demonstrate well the gain in accuracy that
can potentially be gained by using the greedy viscosity technique
described in this paper.

 \begin{figure}[h]
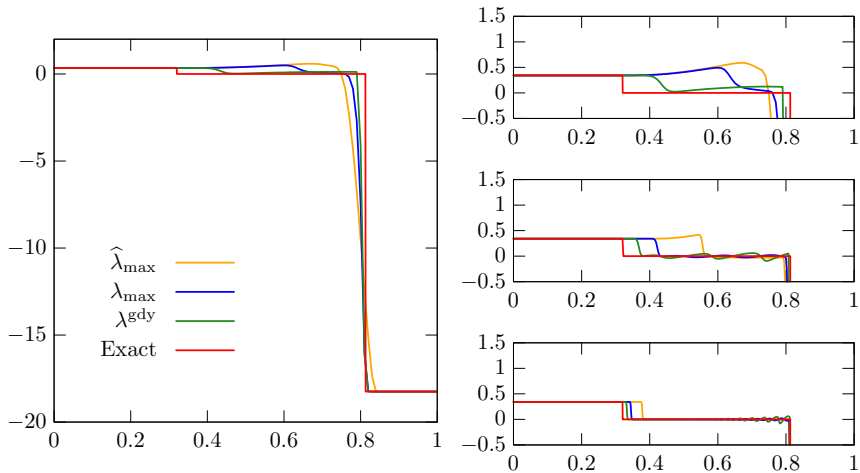

    \begin{minipage}[c]{0.5\textwidth}
\adjustbox{max width=\textwidth}{\input{p_system.tkz}}
    \end{minipage}\hspace{-025pt}
    \begin{minipage}[c]{0.5\textwidth}
\adjustbox{max width=\textwidth}{\input{p_system_100.tkz}}
\adjustbox{max width=\textwidth}{\input{p_system_400.tkz}}
\adjustbox{max width=\textwidth}{\input{p_system_1600.tkz}}
    \end{minipage}
    \caption{Approximation of the $u$ component in the $p$ system,
      $t=0.5$. Left: comparisons between the methods using
        $\wlambda_{\max}$, $\lambda_{\max}$, and
        $\lambda^{\textup{gdy}}$ with $101$ grid points. Right:
        Three refinements: $101$ grid points (top), $401$ grid points  (middle), $1600$ grid points 
        (bottom).} \label{Fig:p_system}
  \end{figure}

  \section{Conclusions}\label{Sec:conclusions}
  We have presented a general strategy to compute the artificial
  viscosity in first-order approximation methods for hyperbolic
  systems.  The technique is based on the estimation of a minimum wave
  speed guaranteeing that the approximation satisfies predefined
  invariant-domain properties and predefined entropy
  inequalities. This approach eliminates non-essential fast waves from
  the construction of the artificial viscosity, while preserving
  pre-assigned invariant-domain properties and entropy inequalities.
  One should however keep in mind that being invariant-domain
  preserving is in general not enough to have a method that is
  robust. Likewise ensuring only one entropy inequality is not a
  guarantee of robustness.

  We finish by briefly demonstrating the performance of the
  proposed methodology when applied to the compressible Euler
  equations.  For each pair $(i,j)$, $i,\in\calV$, $j\in\calI(i)$,
  the greedy viscosity is computed by first computing a wave speed
  that guarantees that the density satisfies local lower and upper
  bounds extracted from the local Riemann problem. This wave speed
  is then augmented by making sure that the specific entropy
  satisfies a local bound.  Finally, the wave speed is possibly
  again augmented to guarantee a local entropy inequality. The
  details are reported in a forthcoming second part of this work. As a
  preview, we consider the Woodward-Colella blast wave problem
  \citep{Woodward_Colella_1984}. We show in
  Figure~\ref{fig:blast_density_GMS_vs_Greedy_ALE} the density profile at
  $t=0.038$. We compare for three different mesh sizes the
  results obtained with the wave speed $\lambda^{\max}$ (labelled
  with the acronym ``GMS'' for guaranteed maximum speed) with those
  obtained with the greedy wave speed $\lambda^{\text{greedy}}$
  (labelled with the acronym ``Greedy''). The superiority of the
  greedy viscosity over the low-order standard method is evident,
  particularly in the region of the leftmost contact wave.
  \begin{figure}[ht]
  \includegraphics[width=0.32\textwidth,trim=275 14 265 14,clip=]{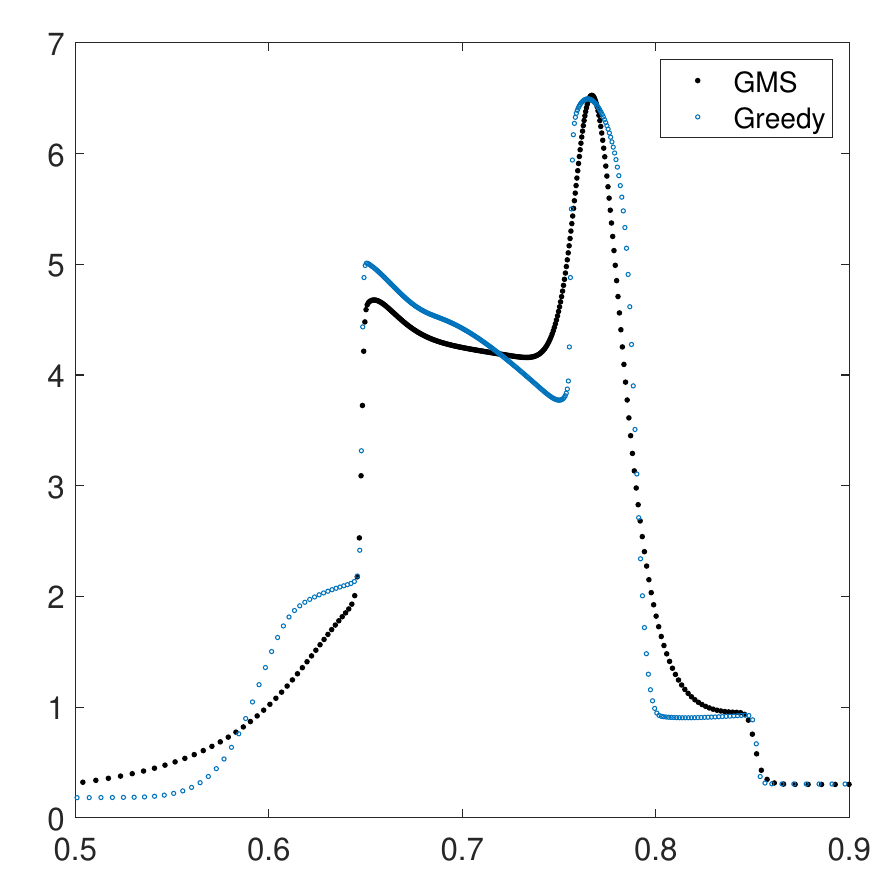}\hfil
  \includegraphics[width=0.32\textwidth,trim=275 14 265 14,clip=]{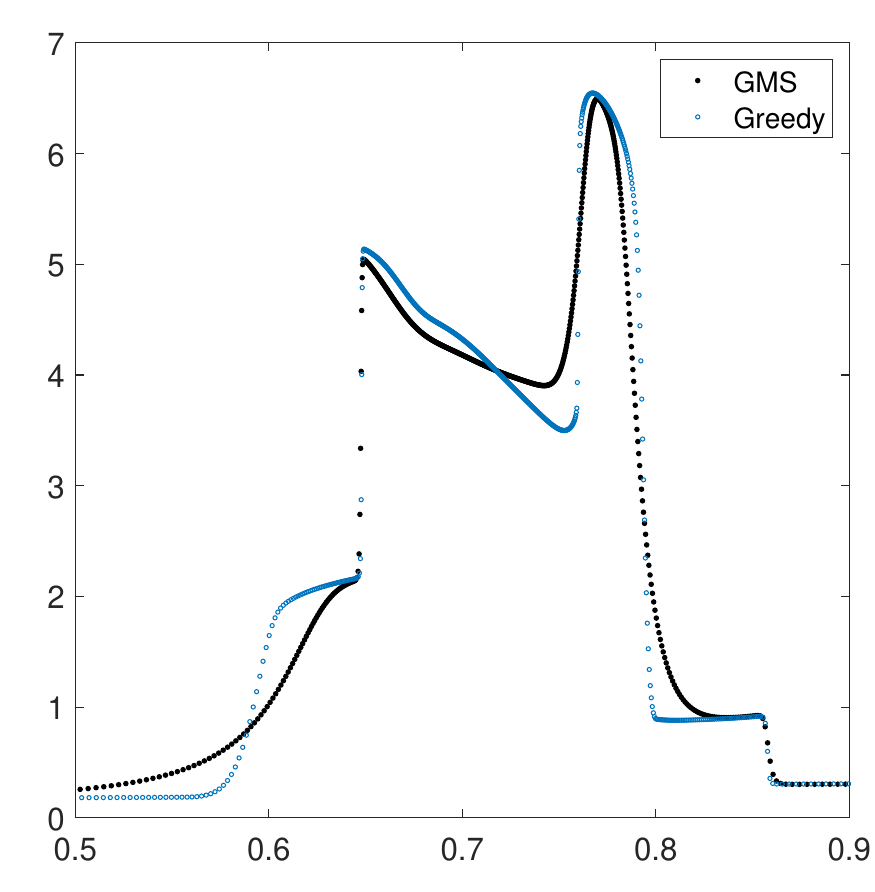}\hfil
  \includegraphics[width=0.32\textwidth,trim=275 14 265 14,clip=]{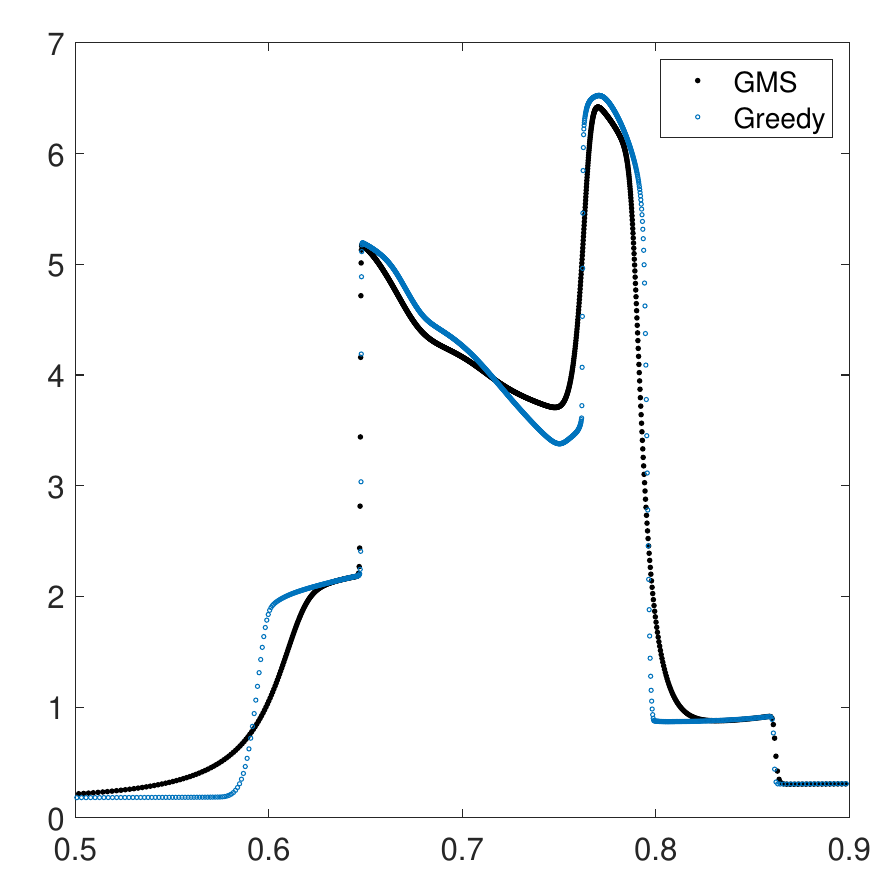}
  \caption{Woodward-Colella blast wave. Density at
  $t=0.038$. GMS vs. Greedy viscosity, from left to right: $\Nglob=400, 800, 1600$.}
  \label{fig:blast_density_GMS_vs_Greedy_ALE}
  \end{figure}

  \section*{Data availability statement}
  Two codes have been written for this project: one in Fortran2018 and
  one in C++. The Fortran code is available upon request. The C++ code is
  part of the Ryujin library
  \url{https://github.com/conservation-laws/ryujin}

  \section*{Funding}
  This material is based upon work supported in part by the National
  Science Foundation grants DMS-1619892 and DMS2110868 (JLG, BP), DMS-1912847 (MM),
  DMS-2045636 (MM), by the Air Force Office of
  Scientific Research, under grant/contract number FA9550-23-1-0007
  (JLG, MM, BP), the Army Research Office, under grant number W911NF-19-1-0431
  (JLG, BP), the U.S. Department of Energy by Lawrence Livermore National
  Laboratory under Contracts B640889, B641173 (JLG, BP), and by the
  European Union-NextGenerationEU, through the National Recovery and
  Resilience Plan of the Republic of Bulgaria, project No BG-RRP-2.004-0008
  (BP), and by the Spanish MCIN/AEI(10.13039/501100011033) under the grant
  PID2020-114173RB-I00 (LS).
  The authors acknowledge the Texas Advanced Computing Center (TACC) at
  The University of Texas at Austin for providing HPC resources that have
  contributed to the research results reported within this paper.
  \url{https://www.tacc.utexas.edu}.

\bibliographystyle{abbrvnat}

\end{document}